\documentclass[12pt, a4, epsf,colorlinks]{amsart} 
\usepackage{scrextend}
\usepackage{epsf, amsmath, amssymb, graphicx, epsfig, amsthm, mathtools}

\usepackage{subfigure} 
\usepackage{setspace} 
\usepackage{fancyhdr} 
\usepackage{eurosym}  
\usepackage{bbm}
\usepackage{verbatim}
\usepackage[tmargin=3.0cm, bmargin=3.0cm, lmargin=3.0cm, rmargin=3.0cm]{geometry}
\usepackage{float}
\usepackage[utf8]{inputenc}
\usepackage{afterpage}
\usepackage{filecontents} 
\usepackage{pgf,tikz,pgfplots}
\usetikzlibrary{arrows.meta}
\usepackage{mathrsfs}
\usetikzlibrary{arrows}

\usepackage{color}
\usepackage{xcolor}
\usepackage{hyperref}

\newcommand*{\ppo}{\textcolor{blue}}

\theoremstyle{plain}
\newtheorem{theorem}{Theorem}[section]
\newtheorem{definition}[theorem]{Definition}
\theoremstyle{definition}

\newtheorem{remark}[theorem]{Remark}
\newtheorem{lemma}[theorem]{Lemma}
\newtheorem{proposition}[theorem]{Proposition}
\newtheorem{lemmaf}[theorem]{Lemma}
\newtheorem{corollary}[theorem]{Corollary}

\newcommand{\E}{\ensuremath{\mathbb{E}}}    
\newcommand{\R}{\ensuremath{\mathbb{R}}}

\newcommand{\N}{\ensuremath{\mathbb{N}}}

\newcommand{\Rd}{\mathbb{R}^d}
\newcommand{\Gd}{\hat{G}_D}
\newcommand{\PhD}{{\hat{P}^D}}
\newcommand{\PD}{P^D}
\newcommand{\pD}{p_D}
\newcommand{\phD}{\hat{p}_D}
\newcommand{\pha}{\hat{p}}
\newcommand{\PP}{\ensuremath{\mathbb{P}}} 
\newcommand{\1}{\ensuremath{\mathbbm{1}}}
\newcommand{\dX}{\Delta X}

\newcommand{\LL}{\mathbb{L}}

\definecolor{rvwvcq}{rgb}{0.08235294117647059,0.396078431372549,0.7529411764705882}
\newcommand{\cadlag}{c\grave{a}dl\grave{a}g}

 \title[ Shot-down stable processes ]{ Shot-down stable processes}

 \author[K. Bogdan ]{ Krzysztof Bogdan }
 \address{
 Krzysztof Bogdan \newline
 \indent Department of Pure and Applied Mathematics \newline
 \indent Wroc\l aw University of Science and Technology, Wroc\l{}aw, Poland
 }
 \email{krzysztof.bogdan@pwr.edu.pl }

 \author[K. Jastrz\c{e}bski]{Kajetan Jastrz\c{e}bski }
 \address{
Kajetan Jastrz\c{e}bski \newline
 \indent Institute of Mathematics \newline
 \indent University of Wroc\l aw, Wroc\l aw, Poland
 }
 \email{kajetanjastrzebski6@gmail.com}

 \author[M. Kassmann]{Moritz Kassmann }
 \address{
Moritz Kassmann \newline
 \indent  Fakult\"{a}t f\"{u}r Mathematik\newline
 \indent Universit\"{a}t Bielefeld,
 Germany
 }
 \email{moritz.kassmann@uni-bielefeld.de}

  \author[M. Kijaczko]{ Micha\l{} Kijaczko }
 \address{
Micha\l{} Kijaczko \newline
 \indent Department of Pure and Applied Mathematics \newline
 \indent Wroc\l aw University of Science and Technology, Wroc\l{}aw, Poland
 }
 \email{michal.kijaczko@pwr.edu.pl}

 \author[P. Pop\l{}awski]{ Pawe\l{} Pop\l{}awski }
 \address{
Pawe\l{} Pop\l{}awski}
 \email{p.poplawski@protonmail.com}
\subjclass[2020]{
 35R09,   	
 31C25 (primary),   	
 60J35,   	
 60J76 (secondary)
 }

\thanks{K. Bogdan was partially supported by NCN grant 2017/27/B/ST1/01339. P. Pop\l{}awski and M.~Kijaczko were partially supported by the NCN grant 2014/14/M/ST1/00600. M. Kassmann gratefully acknowledges support by (a) the Foundation for Polish Science in form of a Alexander von Humboldt Polish Honorary Research Fellowship and (b) the German Science Foundation via CRC 1283.}
\keywords{shot-down process, fractional Laplacian, Dirichlet form, Green function, Harnack inequality. \textit{Data sharing:} not applicable as no data were generated or analysed during the study.}

\begin{document}

\begin{abstract}
The shot-down process is a strong Markov process which is annihilated, or shot down, when \textit{jumping over} or to the complement of a given open subset of a~vector space. Due to specific features of the shot-down time, such processes suggest new type of boundary conditions for 
nonlocal differential equations.
In this work we construct the shot-down process for the fractional Laplacian  in  Euclidean space. For smooth bounded sets $D$, we study its transition density and characterize Dirichlet form. We show that the corresponding Green function is comparable to that of the fractional Laplacian with Dirichlet conditions on $D$.
However,
for nonconvex $D$, the transition density of the shot-down stable process is incomparable with the 
Dirichlet heat kernel of the fractional Laplacian for $D$. 
\end{abstract}

\maketitle
\tableofcontents


\section*{Introduction}\label{Intro} 
Let $D$ be a nonempty open subset of $\R^d$.
The study of diffusions in $D$
naturally leads to the question what happens when the process reaches the boundary, $\partial D$. There are various options and nomenclatures, including \textit{killing}, \textit{absorption}, \textit{censoring}, \textit{resetting}, \textit{resurrection}, \textit{reflection}, \textit{diffusion along the boundary}, etc., which lead to diversified boundary value problems. The situation is similar for general Markov (jump) processes $X=(X_t, t \geq 0)$. Analysis of these problems has generated a lot of research in partial differential equations, potential theory and stochastic analysis, see, e.g., \cite{MR4176673}. 
The simplest option is to \textit{kill} the Markov process at the first exit time of $D$,
$$\tau_D := \inf{} \{ t 
>0: X_t \in D^c\},$$ 
where $D^c=\Rd\setminus D$. The \textit{killed process} on $D$ is then defined by
$$X^D(t) := \begin{cases} X_t & \mbox{for } t\in [0, \tau_D), \\ \partial & \mbox{for } t\in  [\tau_D,\infty). \end{cases} $$
Here $\partial$ is an arbitrary isolated point attached to $\Rd$, called \textit{cemetery} and we extend functions defined on $\Rd$ or $D$ by letting $f(\partial)=0$.

The aim of this work is to introduce and to study a new variant of killing.
To this end we define the \emph{shot-down time}, 
\begin{align*}
\sigma_D = \inf{} \{ t > 0: [X_{t^{-}}, X_t] \cap D^c \neq \emptyset \}\,.
\end{align*}
Here, as usual,  $X_{t^{-}}=\lim_{s\to t^{-}} X_s$, $X_{0-}=X_0$, and $[v,w]$ is the line segment between $v \in \R^d$ and $w \in \R^d$. Trivially, $\sigma_B \leq \sigma_D$ if $B\subset D$, $\sigma_D \leq \tau_D$, 
and $\sigma_D =  \tau_D$ if $D$ is convex. Further, if $D'$ is a connected component of $D$ containing $X_0$, then $\sigma_D=\sigma_{D'}$, therefore below we sometimes assume that $D$ is a \textit{domain}, i.e., a nonempty connected open subset of $\Rd$. Given $x \in D$, we define $D_x = \{y \in D : [x,y]\subset D \}$. 
Only jumps from $x \in D$ to $D_x$ are possible without shooting-down. 
\begin{figure}[h]
	\centering
		\hspace*{4ex}\includegraphics[width=0.3\textwidth]{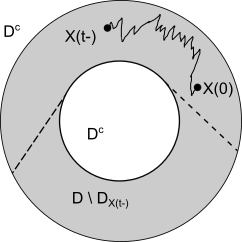}
		\caption{$D$
is an annulus in $\R^2$. The process $X$ is shot down when $[X_{t^{-}}, X_{t}]$ intersects the complement of $D$.}
		\label{fig:gull}
\end{figure}

Analogous to the killed process $X^D$, the \emph{shot-down process} $\hat{X}^D$ is defined by
$$\hat{X}^D(t) := \begin{cases} X_t & \mbox{for } t\in [0, \sigma_D), \\ \partial & \mbox{for } t \in [\sigma_D,\infty). \end{cases} $$
In particular, $\hat{X}^D(t)$ can neither visit $D^c$ nor $D\setminus D_{X_{t^-}}$, for $t\ge 0$.

So far we have not specified to which jump processes $X$ we apply the shooting-down procedure. In principle the construction is very general, but in this work we focus on the isotropic $\alpha$-stable L\'evy process 
\cite{MR3185174}.
The process is associated with the fractional Laplacian $\Delta^{\alpha/2}:=-(-\Delta)^{\alpha/2}$ as generator, see, e.g., \cite{MR3613319}. The resulting shot-down process will sometimes be called \textit{shot-down stable}. For many results, we also restrict ourselves to bounded smooth ($C^{1,1}$) sets $D$, 
see Definition~\ref{def:C11}. 
We leave largely open 
the case of less regular open sets, even Lipschitz, and more general 
Markov processes killed at the shot-down time.

\smallskip

The paper is organized as follows. In the Section \ref{sec:Preli} we present necessary definitions and elementary facts concerning the shot-down process. 
In Section~\ref{sec:Hunt} we define the heat kernel of the process 
and discuss the relationship between the heat kernels of the killed process and the shot-down process. 
In Section~\ref{sec:lk_bound}  we compare the killing measures of the two processes.
In Section \ref{sec:quad-forms} we compute the corresponding Dirichlet form. Sharp estimates of the Green function of the shot-down process are provided in Section~\ref{sec:Green}, which ends the paper.

\smallskip

Let us describe the results in more detail. One important aim of our work is to study the heat kernel $\phD$ of the shot-down process.  To this end, in Theorem \ref{thm:stopping-time} we first assert  that $\sigma_D$ is a stopping time  for general open sets $D$. In Theorem \ref{thm:symmetry} we show that $\phD$ is symmetric in the space variables. 
In Theorem \ref{thm:inc} we prove for bounded smooth $D$ that $\phD$ is comparable to the heat kernel $\pD$ of the killed process if and only if $D$ is convex. 

Then we study the quadratic form of the shot-down process for bounded $C^{1,1}$ domains $D$. For this, we first estimate the intensity of shooting down. Recall that, given $x \in D$, 
$D_x$ consists of all the points where the process can possibly jump from $x$ without being shot down. Denote
\begin{align}\label{eq:alpha-stable-levy-density}
  \nu(z)=
  \mathcal{A}_{d,\alpha}
|z|^{-d-\alpha}\,, \quad z\in \Rd\,,
\end{align}
where $\mathcal{A}_{d,\alpha} =  2^{\alpha}\Gamma((d+\alpha)/2)\pi^{-d/2}/|\Gamma(-\alpha/2)|$.
Since $\nu(y-x)$, $x,y \in \R^d$, is the jump intensity of the process $X$ from $x$ to $y$,  the shooting-down intensity for $D$ is defined by 
$$\iota_D(x) := \int\limits_{\Rd\setminus D_x} \nu(y-x)\,dy, \enspace x\in D.$$ In Theorem \ref{thm:intens-diff-bd} we prove a bound for the difference between the shooting-down intensity and the standard killing intensity for $D$, 
$$\kappa_D(x) := \int\limits_{\Rd\setminus D} \nu(y-x)\,dy, \enspace x \in D.$$
The result is used to establish a representation of the corresponding Dirichlet form in Theorem \ref{thm:Dirichlet-form} and to give a precise description of the domain of the form.

We then move on to the Green function. Here the shot-down stable process behaves like the killed stable process and the two respective Green functions are comparable for every bounded smooth domain $D$, which is proved in Theorem \ref{thm:Green-func-estim}. As mentioned before, such comparability in general fails for the corresponding heat kernels (Theorem \ref{thm:inc}); yet both results build on the study of $\phD$ in Section~\ref{sec:Hunt}. 
Our proofs use potential-theoretic tools typical of jump-type Markov processes, such as the Ikeda-Watanabe formula, but also employ a chaining method more common in the analysis of diffusions. In fact, certain aspects of the shot-down process—its dependence on connectedness and geometry, for instance—resemble diffusions more than jump processes. The analysis of the corresponding harmonic functions is also subtle; we make preliminary steps in this direction.

As a motivation for our study, we mention that there is considerable interest in applied sciences on modeling physical motion by \textit{L\'{e}vy flights}. For some microscopic and macroscopic objects, e.g., bacteria or sharks, it is observed that they migrate by trajectories that have straight stretches. The notion of shot-down processes should be relevant for such studies. To our best knowledge, the mathematical model of shot-down processes is new, but a related concept of \emph{visibility constrained jump process} was introduced in \cite{MR4437349}. As in our work, the visibility constrained jump process can jump from $x \in D$ only to $D_x \subset D$. However, the emphasis of \cite{MR4437349} is on functional inequalities rather than stochastic processes and the process is \textit{censored} rather then killed, should the line segment $[X_{t^{-}}, X_t]$ intersect the complement of $D$. Namely, the process  in \cite{MR4437349} is \textit{restarted} at $X_{t^{-}}$, so the \textit{censoring} is meant in the sense of \cite{MR2006232}.

Concluding the Introduction, we like to mention various approaches in literature to define \emph{reflection} from $D^c$  for jump processes and nonlocal operators, because reflection also uses the process $(X_{t^{-}},X_t)$; see
\cite{MR3217703}, \cite{MR3651008}, \cite{MR4245573} and \cite[Introduction]{MR4779596}. 

\subsection*{Acknowledgements} We are indebted to Mateusz Kwaśnicki for the discussion of the mean value property for the shot-down process.

\smallskip

\smallskip


\section{Preliminaries}\label{sec:Preli} 

In what follows $\Rd$ is the Euclidean space of dimension $d \geq 1$, $x \cdot y$ is the Euclidean scalar product of $x,y\in \Rd$, and $|y|$ is the length of $y$. For $x \in \Rd$ and $A,B \subset \Rd$, we let $\text{dist}(x,A) =\inf\{ |x-y|:y \in A \}$ and $\text{dist}(A,B) = \inf \{ |y-z|: y \in A, z \in B \}$. We define $A^c = \Rd \setminus A$ and, 
for $r>0$, $B_r(x)=B(x,r)=\{z\in \Rd: |z-x|<r\}$. 
\begin{definition}\label{def:C11}
An open set $D \subset \Rd$ is $C^{1,1}$ at scale $r>0$ if for every $Q \in \partial D$ there are balls $I = B(x', r) \subset D $ and $O = B(x'', r) \subset D^c$ tangent at $Q$. We call $I$ and $O$
the \emph{inner} and \emph{outer} ball, respectively and we call $r$ localization radius.
\end{definition}

A bounded open set $D \subset \Rd$ is $C^{1,1}$ if and only if its boundary can be represented locally as the graph of a function whose gradient is Lipschitz continuous \cite[Def. 1.1 and Lemma 2.2]{MR2286038}. Similarly, open set $D\subset \Rd$ is Lipschitz if its boundary is locally equal to the graph of a Lipschitz function\ppo{.}

All our functions, measures and sets are Borel either by construction or assumptions.
As usual, $dy$, $dx$ etc. stand for the Lebesgue measure on $\Rd$, and the considered integrals are assumed to be well defined, to wit, nonnegative or absolutely convergent.

Let  $0<\alpha<2$ and let $X=(X_t,t\geq 0)$ be the standard isotropic $\alpha$-stable L\'evy process in $\Rd$ \cite{MR3185174, MR3613319}. In particular, $X$ is right continuous with left limits.
The process is determined by the jump (L\'evy) measure with the density function \eqref{eq:alpha-stable-levy-density}.
The constant in \eqref{eq:alpha-stable-levy-density} is so chosen that
$$\int \limits_{\Rd} [1- \cos(\xi \cdot z)] \nu (z)\,dz = |\xi|^\alpha, \quad \xi \in \Rd.$$
For every $t>0$, we consider the continuous probability density $p_t: \Rd \to (0,\infty)$ with
$$\int \limits_{\Rd} p_t(x)e^{ix \cdot \xi}\,dx = e^{-t|\xi|^\alpha}, \quad \xi \in \Rd.$$
\noindent
Namely, we let
$$p_t(x) = (2\pi)^{-d}\int \limits_{\Rd} e^{-ix\cdot\xi} e^{-t|\xi|^\alpha}\,d\xi, \quad t>0, \ x \in \Rd.$$
\noindent
The function has the {\it scaling} property:
$$p_t(x) = t^{-d/\alpha}p_1(t^{-1/\alpha}x), \quad \text{ or }  \quad p_{r^\alpha t}(rx) = r^{-d}p_t(x), \quad t>0,\ r>0,\ x\in \Rd.$$
Furthermore, $p_t$ form a convolution semigroup of probability densities.
It is well known that for $\alpha  = 1$,
$$p_t(x) = \frac{C_dt}{(t^2 + |x|^2)^{(d+1)/2}}, \quad t>0, \ x \in \Rd,$$ where $C_d = \Gamma((d+1)/2)/\pi^{(d+1)/2}$, and, for any  $\alpha\in (0,2)$,
\begin{equation}
    \label{pt:approx}
    p_t(x) \approx t^{-d/\alpha} \land \frac{t}{|x|^{d+\alpha}}, \quad t>0, \ x \in \Rd.
\end{equation}
Here $\land$ stands for the minimum and \eqref{pt:approx} means that there is a number $c\in (0,\infty)$ (i.e., a \textit{constant}, depending on $d$ and $\alpha$)  such that  $c^{-1}p_t(x) \leq t^{-d/\alpha} \land \frac{t}{|x|^{d+\alpha}} \leq cp_t(x)$. This notation for \textit{comparability} is used throughout the paper.
We denote $$p(t,x,y)=p_t(x,y)=p_t(y-x).$$ The process $X_t$ is Markovian
with the time-homogeneous transition probability
$$(t,x,A) \mapsto \int \limits_{A} p(t,x,y)\,dy, \quad t>0,\, x \in \Rd, A \subset \Rd.$$
We let $\PP_x$ and  $\E_x$ be the law and expectation for the process starting at $x\in \Rd$, respectively.
\noindent
We consider the operator semigroup $\{P_t, t\geq0 \}$ on bounded or nonnegative function $f$.
Thus, for $x\in \Rd$, $P_t f(x) =\E_x f(X_t)$ and
$$P_t f(x) = \int \limits_{\Rd} p(t,x,y)f(y)\,dy, \quad t>0,$$
see ~\cite[Section 2]{MR2794975} for a direct construction of the semigroup.
\noindent
We will be interested in semigroups obtained by {\it killing} $X$ at suitable stopping times, chiefly at the shot-down time for open sets $D \subset \Rd$.
\noindent
Recall that $X$ is defined on the measurable space $\Omega$ of c\`adl\`ag paths: $[0,\infty)\to \Rd$  \cite{MR3185174}, so that $X_t=X_t(\omega)$ for $t>0, \omega \in \Omega$. As usual, we suppress $\omega$ from the notation. 

In the Introduction we defined $\tau_D$, the first exit time from $D$, $X^D$, the process killed upon exiting  $D$, $\sigma_D$, the shot-down time,  and $\hat{X}^D$, the shot-down process.
We now introduce the corresponding operator semigroups,
$$P^D_tf(x) = \E_x[t<\tau_D; f(X_t)]$$
and 
$$\hat{P}^D_tf(x) = \E_x[t<\sigma_D; f(X_t)], \quad t\ge 0,\, x \in \Rd.$$
Here, as usual, $f$ is nonnegative or absolutely integrable and $\E[A; F]=\int_A F d\PP$.

\subsection{Progressive measurability}\label{subsec:meas} 
The first time the process $X$ flies over the complement of $D$, that is the shot-down time, is indeed a stopping time. The proof depends on the debut theorem for progressively measurable processes.

\begin{theorem}\label{thm:stopping-time}$\sigma_D = \inf{} \{ t {>} 0: [X_{t^{-}}, X_t] \cap D^c \neq \emptyset \}$ is a stopping time.
\end{theorem}
\begin{proof}
Of course, $X$ is $c \grave {a}dl \grave{a}g$ and $X_{t^{-}}$ is $c \grave {a}gl \grave{a}d$. By \cite[Theorem 1, p. 38]{MR648601}, each process is progressively measurable, and so is
the pair $(X_t, X_{t^{-}})$.
We easily see that $F(D):=\{ (x,y) \in \Rd \times \Rd: [x,y] \cap D^c \neq \emptyset \}$
is closed, so Borel. Of course, the first {hitting of $F(D)$ by} $(X_t, X_{t^{-}})$  is $\sigma_D$, hence $\sigma_D$ is a stopping time by \cite[Theorem 2.{4}]{MR2606507}; see also \cite{MR2788890} and \cite[Subsection 1.1]{bogdan2023shotdown}.
\end{proof}

\subsection{Ikeda-Watanabe formula}\label{IWF}

For the rest of the paper, we return to the  setting of the fractional Laplacian in the Euclidean space $\Rd$. Let $D\subset \Rd$ be open (we will make additional assumptions on $D$ later on).
In this section we prove the Ikeda-Watanabe-type formula \eqref{gwiazdka} for the shot-down stable processes. To this end we use the so-called L\'evy system for $X$ and follow the presentation for the killed process in \cite[Section~4.2]{MR3737628}.
Note that $\PhD_t \mathbbm{1}_A(x) \leq \PD_t \mathbbm{1}_A(x)$ for all $A \subset \Rd$. By the Radon-Nikodym theorem, for all $x \in \Rd,\ t>0$ and functions $f \geq 0$, we can write
$$\PhD_t f(x) = \int \limits_{\Rd} \phD(t,x,y)f(y)\,dy,$$
where $\phD(t,x,y)$ is defined  for almost all $y$ and $\phD(t,x,y) \leq \pD(t,x,y) \leq p_t(x,y)$. Here $\pD$ is the Dirichlet heat kernel of $D$ for $\Delta^{\alpha/2}$, discussed in Section~\ref{sec:Hunt}.
We may call $\phD$ the heat kernel of the shot-down process, because
\begin{equation} \label{eq:int1}
\E_x \int_0^{\sigma_D} f(t,X_t)\,dt=\int_0^\infty \int_{\Rd} \phD(t,x,y)f(t,y)\,dy\,dt.
\end{equation}

\begin{theorem}\label{I-W}
The joint distribution of $(\sigma_D,X_{\sigma_{D}^{-}},X_{\sigma_D})$ restricted to $\{X_{\sigma_{D}^{-}}\in D\}$ and calculated under $\PP_x$ for $x \in D$, satisfies for $A \subset D, B \subset (\overline{D})^c, I \subset [0,\infty)$,
\begin{equation} \label{gwiazdka}
\mathbb{P}_x[\sigma_D \in I, X_{\sigma_{D}^{-}} \in A, X_{\sigma_D} \in B] = \int \limits_I \int \limits_A\int \limits_B\,\nu(w-y) \phD(t,x,y)\,dw\,dy\,dt.   
\end{equation}

\end{theorem}
\begin{proof}
In view of the monotone convergence theorem, we may assume that $I$ and $B$ are compact.
Let $F(u,y,w)=\1_{I}(u)\1_A(y) \1_B(w)$ and
\begin{align*}
M(t)= \sum\limits_{\substack{0<u \le t \\ |\dX_u| \neq 0}}F(u, X_{u^-}, X_{u}) - \int_{0}^{t} \int_{\Rd} F(u,X_{u},
X_{u}+z)\nu(z)\,dz\,du, \textrm{     } 0 \leq t < \infty.
\end{align*}
By the L\'evy system, {see, e.g., } \cite[Lemma~4.1]{MR3737628},
\begin{align*}
   \E_x  \sum\limits_{\substack{0<u \le \infty \\ |\dX_u| \neq 0}}F(u, X_{u^-}, X_{u}) &= \E_x \int_{0}^{\infty} \int_{\Rd} F(u,X_{u},X_{u}+z)\nu(z)\,dz\,du.
\end{align*}
Thus, $\E_x M(t)=0$. In fact, $M(t)$ is a martingale. Indeed, let $0\le s\le t$. By considering the L\'evy process $u\mapsto X_{s+u}-X_s$, independent of $ \{X_r$, $0  \le r \le s \}$, we calculate the conditional expectation:
$$\E_x \bigg [\sum\limits_{\substack{s<u \le t \\ |\dX_u| \neq 0}}F(u, X_{u^-}, X_{u}) - \int_{s}^{t} \int_{\Rd} F(u,X_{u},
X_{u}+z)\nu(z)\,dz\,du\ \bigg|\  X_r, 0\le r\le s \bigg ]=0.$$
Since 
\begin{align*}
   |M(t)|\leq \sum\limits_{\substack{0< u < \infty \\ |\dX_u| \neq 0}}F(u, X_{u^-}, X_{u}) + \int_{0}^{\infty}  \int_{\Rd} F(u,X_{u},X_{u}+z)\nu(z)\,dz\,du,
\end{align*} 
and the right hand side has expectation not bigger than 
$2|I| \nu ( \{ |z| \geq {\rm dist}(A,B) \}) < \infty $, we see that $M$ is a uniformly integrable $c \grave {a}dl \grave{a}g$ martingale. By stopping the martingale at $\sigma_D$,   
we obtain ~\cite[Section 12.5]{MR2059709}
$$ \E \sum\limits_{\substack{0 <u < \sigma_D \\ |\dX_u| \neq 0}}F(u, X_{u^-}, X_{u}) = \E \int_{0}^{\sigma_D}  \int_{\Rd} F(u,X_{u},X_{u}+z)\nu(z)\,dz\,du.$$
By \eqref{eq:int1}, this amounts to 
\begin{align}\label{eq:IW}
\PP^x[\sigma_D\in I,\; X_{\sigma_D^{-}}\in A,\; X_{\sigma_D}\in B]=
\int_I \int_{B-y} \int_A \phD(u,x,y)\nu(z)\,dy\,dz\,du.\qquad \qedhere
\end{align}
\end{proof}

\begin{lemma}\label{l.nhb}
If open set $D\subset \Rd$ is Lipschitz, then $\PP^x[X_{\sigma_D^{-}}\in D]=1$, $x\in D$.
\end{lemma}
\begin{proof}
It is well known that $X$ is quasi-left continuous \cite{MR3185174}.
Let $X_0=x\in D$.
If $X_{\sigma_D^{-}}\in \partial D$, then, by the quasi-left continuity and the proof of \cite [Lemma 17]{MR1438304}, $\sigma_D\le \tau_D$, hence $\sigma_D=\tau_D$ and so $X_{\tau_D^{-}}\in \partial D$, which has probability $0$, by \cite [Lemma 17]{MR1438304}.
\end{proof}
\begin{remark}\label{r.nhb}
{\rm For Lipschitz $D$, Lemma~\ref{l.nhb} allows to apply \eqref{gwiazdka} to arbitrary (Borel) $B\subset D^c$. It would even suffice
to assume that $D$ has the outer cone property or
less (see \cite[Appendix A.1]{MR4088505}), but such geometric restrictions 
would 
linger in our applications of the Ikeda-Watanabe formula. Of course, $C^{1,1}$ open sets are Lipschitz.
For clarity, we also note that in \eqref{gwiazdka}, we have $X_{\sigma_D}\in D^c$, so  $\sigma_D=\tau_D$. For non-convex sets $D$, this is only a part of the possible scenarios for $X_{\sigma_D}$, see, e.g., Theorem~\ref{thm:inc}.}
\end{remark}
\section{The heat kernel  of the shot-down stable process}\label{sec:Hunt} 

Recall that $X$ is the isotropic $\alpha$-stable L\'evy process in $\Rd$, with $\alpha\in (0,2)$ and $d=1,2,\ldots$.  Let $D$ be an arbitrary open set in $\Rd$.\footnote{Starting from Corollary~\ref{fact:huntIke} we  make additional assumptions on $D$.}

G. Hunt's formula defines the transition density $\pD(t,x,y)$ \textit{pointwise}:
\begin{equation}\label{eq:Hfk} 
\pD(t, x, y) = p(t,x,y) - \E_x[\tau_D < t; p(t-\tau_D, X_{\tau_D}, y)], \quad t>0,\, x,y \in \Rd.
\end{equation}
In agreement with Subsection~\ref{IWF}, the following equality holds, 
\begin{equation}\label{eq:stdkp}
\PP_x(t<\tau_D; X_t \in A) = \int \limits_A \pD(t,x,y)\,dy, \quad t>0,\, x\in \Rd, A\subset \Rd,
\end{equation}
see \cite[Theorem 2.4]{MR1329992}, and for nonnegative or absolutely integrable functions $f$, 
$$P^D_t f(x) = \int \limits_{\Rd} \pD(t,x,y)f(y)\,dy , \quad t>0,\, x \in \Rd.$$
We will extend this to the shot-down process.
\subsection{Hunt-type formula for the shot-down process}
In Section~\ref{IWF} we defined the heat kernel of the shot-down process $\phD(t,x,y)$ for $a.e.$ $y$, by using  the Radon-Nikodym theorem. We now define $\phD(t,x,y)$ {\it pointwise} for all $x,y$ in $D$. 
\begin{definition} \label{def_hunt}
	\begin{align*}\label{eq:Hfsd} 
	\phD(t, x, y) = p(t,x,y) - \E_x[\sigma_D < t; p(t-\sigma_D, X_{\sigma_D}, y)], \quad t>0,\, x,y \in \Rd.
	\end{align*}
\end{definition}
\noindent
The following theorem is an analog of the first part of ~\cite[Theorem 2.4]{MR1329992}.
It can be proven in much the same way 
as in \cite{MR1329992}, 
but we give the proof for convenience and completeness. 
\begin{theorem}
	$\PP_x(t<\sigma_D; X_t \in A) = \int \limits_A \phD(t,x,y)\,dy$ for $t>0$, $x\in \Rd$, $A\subset \Rd$.
\end{theorem}

\begin{proof}
	We note that for fixed $t>0$, $\PP_x(\sigma_D=t)=0$. Indeed, $\PP_x(X_t \in \partial D)=0$ and by \cite[(1.10)]{MR3185174} $\PP_x-a.s.$ $X$ is continuous at $t$, which implies that $\PP_x(\sigma_{D}\neq t)=1$.
	Therefore,
	\begin{equation}\label{eq:pcont}
	\begin{split}
	\PP_x(t<\sigma_D; X_t \in A) &= \PP_x(X_t \in A) - \PP_x(\sigma_D \leq t; X_t \in A) \\
	&= \int \limits_A p(t,x,y)\,dy - \PP_x(\sigma_D < t; X_t \in A).
	\end{split}
	\end{equation}
	Let $0<u<t$, $n \geq 1$ and $1\leq k\leq2^n$. We define stopping times
	$$S_n = \begin{cases} ku2^{-n}, & \mbox{if } (k-1)u2^{-n} \leq \sigma_D < ku2^{-n}, \\ \infty, & \mbox{if } \sigma_D \geq u. \end{cases} $$
	By the Markov property \cite{MR648601}, 
	\begin{equation}\label{eq:Pxest}
	\begin{split}
	\PP_x(\sigma_D < u; X_t \in A) &= \sum\limits_{k=1}^{2^n} \PP_x \bigg(\frac{(k-1)u}{2^n} \leq \sigma_D < \frac{ku}{2^n}; X_t \in A \bigg) \\
	&= \sum\limits_{k=1}^{2^n} \E_x \bigg[\frac{(k-1)u}{2^n} \leq \sigma_D < \frac{ku}{2^n}, \PP_{X_{ku2^{-n}}}(X_{t-ku2^{-n}} \in A) \bigg] \\
	&= \E_x \bigg[ \sigma_D<u; \int \limits_A p(t-S_n, X_{S_n}, y)\,dy  \bigg].
	\end{split}
	\end{equation}
	If $\sigma_D < u$, then $t-S_n \geq t-u > 0$  and we get $S_n \downarrow \sigma_D$. We let $n \to \infty$. Then we let $u \uparrow t$ in \eqref{eq:Pxest}. By the right continuity of the process,  joint continuity of $p(t,x,y)$, dominated convergence, Fubini-Tonelli and monotone convergence,
	\begin{align*}
	\PP_x(\sigma_D<t; X_t \in A) &= \lim_{u\uparrow t} \PP_x(\sigma_{D}<u; X_t \in A)\\
	&= \lim_{u\uparrow t} \E_x \left[\sigma_{D}<u;\int_A p(t-\sigma_D,X_{\sigma_D},y)\,dy\right]\\
	&=\E_x \left[\sigma_{D}<t;\int_A p(t-\sigma_D,X_{\sigma_D},y)\,dy\right]\\
	&=\int \limits_A \E_x \bigg[ \sigma_D <t; p(t-\sigma_D, X_{\sigma_D}, y) \bigg]\,dy.
	\end{align*}
	By \eqref{eq:pcont},
	\begin{align*}
	\PP_x(t<\sigma_D; X_t \in A) &= \int \limits_A p(t,x,y)\,dy - \int \limits_A \E_x \bigg[ \sigma_D <t; p(t-\sigma_D, X_{\sigma_D}, y) \bigg]\,dy \\
	&=  \int \limits_A \phD(t,x,y)\,dy.\qedhere
	\end{align*}
\end{proof}

The heat kernel of the killed process enjoys the scaling: $r^{\alpha}p_{rD}(r^{\alpha}t,rx,ry)=p_D(t,x,y)$. Here $r>0$, $x\in \Rd$, $rD=\{rx: x\in D\}$. We will prove a similar equality for $\phD$.
\begin{lemmaf}\label{fact_scal}
The law of $\sigma_{rD}$ under $\PP_{rx}$ equals that of $r^{\alpha}\sigma_{D}$ under $\PP_x$.
\end{lemmaf}

\begin{proof}
	
	We have:
	\begin{eqnarray*}
		\sigma_{rD} &=& \inf\left\{t> 0:\left[X_{t^{-}},X_t\right]\cap(rD)^{c}\neq\emptyset\right\},\\
		&=&\inf\left\{t>0: \left[X_{t^{-}}/r, X_t/r\right] \cap D^{c} \neq \emptyset\right\}.
	\end{eqnarray*}
	The distribution of $X_t/r$ when $X$ starts at $rx$ is the same as the distribution of $X\left(t/r^{\alpha}\right)$ when $X$ starts at $x$.  In the same sense
	$X_{t^{-}}/r$ and $X\left(t/r^{\alpha}-\right)$ equal in distribution,
	and the last infimum above equals in distribution to
	\begin{displaymath}
	\inf\left\{t>0:\left[X\left(t/r^{\alpha}-\right), X\left(t/r^{\alpha}\right)\right] \cap D^{c} \neq \emptyset \right\}=r^{\alpha}\sigma_{D}. \qedhere
\end{displaymath}
\end{proof}

\begin{lemma}\label{lem:schk}
	$r^{d}\hat{p}_{rD}(r^{\alpha}t,rx,ry) = \phD(t,x,y)$ for $r>0$, $x \in \Rd$.
\end{lemma}
\begin{proof}
	By the Hunt formula in Definition \ref{def_hunt} and Lemma~\ref{fact_scal}, 
	\begin{align*}
	r^{d}\pha_{rD}(r^{\alpha}t,rx,ry)&=r^{d}{p}(r^{\alpha}t,rx,ry)-r^{d}\mathbb{E}_{rx}[\sigma_{rD}<r^{\alpha}t,p(r^{\alpha}t-\sigma_{rD},X_{\sigma_{rD}},ry)]\\
	&={p}(t,x,y)-r^{d}\mathbb{E}_{x}[r^{\alpha}\sigma_{D}<r^{\alpha}t;p(r^{\alpha}(t-\sigma_{D}),X_{r^{\alpha}\sigma_{D}},ry)]\\
	&={p}(t,x,y)-\mathbb{E}_{x}[\sigma_{D}<t;r^{d}p(r^{\alpha}(t-\sigma_{D}),rX_{\sigma_{D}},ry)]\\
	&={p}(t,x,y)-\mathbb{E}_{x}[\sigma_{D}<t;p(t-\sigma_{D},X_{\sigma_{D}},y)]=\phD(t,x,y).
	\qquad \qedhere
	\end{align*}
\end{proof}

\subsection{Perturbation formula}\label{sec:mHf}
We note the following extension of Hunt's formula.
\begin{lemma}\label{lem:Hffssd}
	Let $U$ be open, $U \subset D$. Then, for $t>0$ and $x,y\in \Rd$ we have
	$$
	\pha_U(t,x,y) = \phD(t,x,y) - \E_x[\sigma_U<t; [X_{\sigma_U-},X_{\sigma_U}]\subset D; \phD(t-\sigma_U,X_{\sigma_U},y)].
	$$
\end{lemma}
\begin{proof}
	We have
	\begin{align*}
	p(t,x,y)-\phD(t,x,y) = \E_x[\sigma_D<t;p(t-\sigma_D,X_{\sigma_D},y)]
	\end{align*}
	and
	\begin{align*}
	p(t,x,y)-\hat{p}_U(t,x,y) = \E_x[\sigma_U<t;p(t-\sigma_U,X_{\sigma_U},y)].
	\end{align*}
	Subtracting we get
	\begin{align*}
	\phD(t,x,y)-\hat{p}_U(t,x,y) 
	&=\E_x\left[\sigma_U<t;p(t-\sigma_U,X_{\sigma_U},y)\right] - \E_x\left[\sigma_D<t;p(t-\sigma_D,X_{\sigma_D},y)\right]\\
	&=\E_x\left[\sigma_U<t,\sigma_U<\sigma_D;p(t-\sigma_U,X_{\sigma_U},y)\right]\\
	&\quad -\E_x\left[\sigma_D<t,\sigma_U<\sigma_D;p(t-\sigma_D,X_{\sigma_D},y)\right]\\
	&=\E_x\left[\sigma_U<t,[X_{\sigma_U-},X_{\sigma_U}]\subset D; p(t-\sigma_U,X_{\sigma_U},y)\right]\\
	&\quad -\E_x\left[\sigma_D<t,[X_{\sigma_U-},X_{\sigma_U}]\subset D; p(t-\sigma_D,X_{\sigma_D},y)\right],
	\end{align*}
	where {the next to last} equality follows since $\sigma_U \leq \sigma_D$ and both integrands are equal on the set $\sigma_U = \sigma_D$. Denote the first expectation by $I_1$ and the second by $I_2$. Let $\mathcal{F}_{\sigma_U}$ be the usual $\sigma$-field of pre-$\sigma_U$ events and $\theta_{\sigma_U}$ be the usual shift, e.g., $\theta_{\sigma_U}X_t=X_{t+\sigma_U}$. By the strong Markov property, since $\sigma_D=\sigma_U+\sigma_D\circ\theta_{\sigma_U}$, we get
	\begin{align*}
	\begin{split}
	I_2&=\E_x\left[\sigma_D<t,\sigma_U<t,[X_{\sigma_U-},X_{\sigma_U}]\subset D; p(t-\sigma_D,X_{\sigma_D},y)\right]\\
	&=\E_x\left[\sigma_U<t,[X_{\sigma_U-},X_{\sigma_U}]\subset D; \E\left[\1_{\sigma_D<t}\ p(t-\sigma_D,X_{\sigma_D},y)|\mathcal{F}_{\sigma_U}\right]\right]\\
	&=\E_x\left[\sigma_U<t,[X_{\sigma_U-},X_{\sigma_U}]\subset D; \E\left[\1_{\sigma_U+\sigma_D\circ\theta_{\sigma_U}<t}\ p(t-\sigma_U-\sigma_D\circ\theta_{\sigma_U}, X_{\sigma_D}\circ\theta_{\sigma_U},y)|\mathcal{F}_{\sigma_U}\right]\right]\\
	&=\E_x\left[\sigma_U<t,[X_{\sigma_U-},X_{\sigma_U}]\subset D; \E_{X_s}\left[\sigma_D<t-s; p(t-s-\sigma_D, X_{\sigma_D},y)\right]\mid_{s=\sigma_U}\right]\\
	&=\E_x\left[\sigma_U<t,[X_{\sigma_U-},X_{\sigma_U}]\subset D;p(t-\sigma_U,X_{\sigma_U},y)-\phD(t-\sigma_U,X_{\sigma_U},y)\right],
	\end{split}
	\end{align*}
	so
	\begin{align*}
	\begin{split}
	I_1-I_2=\E_x[\sigma_U<t,[X_{\sigma_U-},X_{\sigma_U}]\subset D;\phD(t-\sigma_U,X_{\sigma_U},y)].\qquad\qedhere
	\end{split}
		\end{align*}
\end{proof}
The Hunt formula defines the kernel of the killed process and the kernel of the shot-down process in terms of $p$. Here is a direct relationship between  $\pD$ and $\phD$. 
\begin{lemma} \label{lem:nonlocal}
	For $t>0$ and $x,y \in \Rd$, the following equality holds
	$$
	\phD(t, x, y)  =\pD(t,x,y) - \E_x \left[\sigma_D < t, X_{\sigma_D} \in D \setminus D_
	{ X_{\sigma_D-}}; \pD(t - \sigma_D , X_{\sigma_D}, y)\right].
	$$
\end{lemma}
\begin{proof}
	Combining definitions of $\pD$ and $\phD$ for $t>0,\, x,y \in \Rd$, we get
	$$\pD(t,x,y) - \phD(t, x, y)  = \E_x[\sigma_D < t; p(t-\sigma_D, X_{\sigma_D}, y)] - \E_x [\tau_D < t; p(t-\tau_D, X_{\tau_D}, y)]. 
	$$
	We have $\sigma_D \leq \tau_D$ so the last expression equals
	$$\E_x[\sigma_D < t, \sigma_D < \tau_D; p(t-\sigma_D, X_{\sigma_D}, y)] - \E_x [\tau_D < t, \sigma_D < \tau_D; p(t-\tau_D, X_{\tau_D}, y)]. 
	$$
	We write $r^D(t, x, y)=\E_x[\tau_D < t; p(t-\tau_D, X_{\tau_D}, y)]$ to simplify the notation. Making use of the main part of the proof of ~\cite[Proposition 2.3]{MR2438694}, we will show below that
	\begin{equation}\label{eq:Hu}
	\E_x[\tau_D < t, \sigma_D < \tau_D; p(t-\tau_D, X_{\tau_D}, y)] = \E_x [\sigma_D < t, \sigma_D < \tau_D; r^D(t-\sigma_D, X_{\sigma_D}, y)].
	\end{equation}
	Indeed, the strong Markov property yields
	\begin{align*}
	\begin{split}
	\E_x &[\tau_D < t, \sigma_D < \tau_D; r^D(t-\sigma_D, X_{\sigma_D}, y)]\\
	&=\E_x \bigg[\sigma_D < t, \sigma_D < \tau_D; \E_{X_{\sigma_D}}[\tau_D <t-s; p(t-s-\tau_D, X_{\tau_D}, y)] \big|_{s = \sigma_D}  \bigg]\\
	&= \E_x \bigg[\sigma_D < t, \, \sigma_D < \tau_D, \tau_D \circ \Theta_{\sigma_D} + \sigma_D < t; p(t - \tau_D \circ \Theta_{\sigma_D} - \sigma_D, X_{\tau_D} \circ \Theta_{\sigma_D}, y)\bigg].
	\end{split}
	\end{align*}
	On the set $\sigma_D < \tau_D$ we have $\tau_D \circ \Theta_{\sigma_D} + \sigma_D = \tau_D$ and $X_{\tau_D} \circ \Theta_{\sigma_D} = X_{\tau_D}$, so simplifying the last expression we get  
	$$
	\E_x \left[\sigma_D < t, \, \sigma_D < \tau_D; p(t - \tau_D , X_{\tau_D}, y)\right],
	$$
	which ends the proof of \eqref{eq:Hu}. From \eqref{eq:Hu} we obtain
	\begin{align*}\pD(t,x,y) - \phD(t, x, y)
	=\E_x \left[\sigma_D < t, \, \sigma_D < \tau_D; \pD(t - \sigma_D , X_{\sigma_D}, y)\right].
	\qquad\qedhere
	\end{align*}
\end{proof}
Note that the condition $\sigma_D < \tau_D$ can be written as $X_{\sigma_D} \in D \setminus D_{X_{\sigma_D-}}$.
The next (Duhamel) formula shows $\pha_D$ as a \textit{nonlocal perturbation} of $\pD$, in the sense of \cite{MR3295773}. 
\begin{corollary} \label{fact:huntIke}
	For Lipschitz $D$, $x, y \in \Rd$ and $t>0$, the following equality holds
	$$\phD(t, x, y)  = \pD(t,x,y) - \int \limits_0^t \int \limits_{D} \int \limits_{D\setminus D_w} \!\!  \phD(s,x,w) \nu(z-w) \pD(t-s, z,y)\,dz\,dw\,ds.$$
\end{corollary}
\begin{proof}
	Remark~\ref{r.nhb}  
gives the joint probability distribution of $\sigma_D$, $ X_{\sigma_D -}$ and $X_{\sigma_D}$. Substituting this in Lemma~\ref{lem:nonlocal}, we obtain the statement.
\end{proof}

\subsection{The Chapman-Kolmogorov equation} 
	In this section we prove the Chapman-Kolmogorov equation for $\phD(t,x,y)$.
Let $\hat{\nu}(x,y) := \nu(y-x)\1_{D_x^c}(y)$, where
$D_x = \{y \in D : [x,y]\subset D \}$. Of course, $\hat{\nu}(x,y)=\hat{\nu}(y,x)$.
We start with a regularity result.
	\begin{lemma}
		The function $y\mapsto \phD(t,x,y)$ is continuous on $D$ for all $t>0$, $x\in D$.
	\end{lemma}
	\begin{proof}
		Fix $t>0,\, x\in D$. {Corollary \ref{fact:huntIke} states:}
		$$\phD(t,x,y) = \pD(t,x,y) - \int_0^t\int_D\int_D\ \phD(s,x,w)\hat{\nu}(w,z)p_D(t-s,z,y)\,dz\,dw\,ds.$$
		Since $p_D(t,x,\cdot)$ is continuous, it is enough to prove the continuity of the function
		$$
		y\mapsto\int_0^t\int_D\int_D\ \phD(s,x,w)\hat{\nu}(w,z)p_D(t-s,z,y)\,dz\,dw\,ds.
		$$
		To this end, take $D\ni y_n\to y\in D$. Let $\delta=\inf_n \text{dist}(y_n,D^c)>0,\text{ and } A=\{z\in D : \text{dist}(z,D^c)<\delta/2\}$. 
\newline If $z\in A$, then $|z-y_n|\geq\delta/2$ for every $n$, so
		\begin{align*}
		\phD(s,x,w)\hat{\nu}(w,z)p_D(t-s,z,y_n)\1_A(z)\leq C2^{d+\alpha}t\delta^{-d-\alpha} \phD(s,x,w) \hat{\nu}(w,z) \1_A(z).
		\end{align*}
		This majorant is integrable by the Ikeda-Watanabe formula:
		\begin{align*}
		\int_0^t\int_D\int_A \phD(s,x,w)\hat{\nu}(w,z)\,dz\,dw\,ds= \PP_x(\sigma_D<t,X_{\sigma_D}\in A)\leq 1<\infty,
		\end{align*}
		so our integrand converges by the Lebesgue dominated convergence theorem for $n\to \infty.$
		Now, assume $z\in D\setminus A$. We have $B(z,\delta/2)\subset D_z$, thus
		\begin{align*}
		&\phD(s,x,w) \hat{\nu}(w,z) p_D(t-s,z,y_n) \1_{D\setminus A}(z)\\
		&\leq \mathcal{A}_{{d,\alpha}}2^{d+\alpha}\delta^{-d-\alpha} \phD(s,x,w) p_D(t-s,z,y_n) \1_{D\setminus A}(z)\\
		&\leq C\mathcal{A}_{{d,\alpha}}2^{d+\alpha}\delta^{-d-\alpha} \phD(s,x,w)\left((t-s)^{-d/\alpha} \land \frac{t-s}{|z-y_n|^{d+\alpha}}\right)\1_{D\setminus A}(z).
		\end{align*}
		Let $\lambda=\sup_n|y-y_n|$. Then $\inf_n |z-y_n| \geq |z-y| - \lambda,$ so
		\begin{align*}
		&\phD(s,x,w) \hat{\nu}(w,z) p_D(t-s,z,y_n) \1_{D\setminus A}(z)\\
		&\leq \frac{C\mathcal{A}_{{d,\alpha}}2^{d+\alpha}}{\delta^{d+\alpha}} \phD(s,x,w) \left((t-s)^{-d/\alpha} \1_{B(y,2\lambda)}(z) + \frac{2^{d+\alpha}(t-s)}{|z-y|^{d+\alpha}} \1_{B(y,2\lambda)^c}(z) \right) \1_{D\setminus A}(z).
		\end{align*}
		This function is integrable with respect to $dw\,dz$. By the Lebesgue dominated convergence theorem, we get
		\begin{align*}
		&\int_D\int_{D\setminus A} \phD(s,x,w)\hat{\nu}(w,z) p_D(t-s,z,y_n)\,dz\,dw\\
		&\rightarrow \int_D\int_{D\setminus A} \phD(s,x,w)\hat{\nu}(w,z) p_D(t-s,z,y)\,dz\,dw.
		\end{align*}
		Further, by the symmetry of $p_D$, we can write
		\begin{align*}
		&\int_D\int_{D\setminus A} \phD(s,x,w)\hat{\nu}(w,z) p_D(t-s,z,y_n)\,dz\,dw\\
		&\leq \mathcal{A}_{{d,\alpha}}2^{d+\alpha}\delta^{-d-\alpha}\int_D\int_{D\setminus A}\phD(s,x,w) p_D(t-s,y_n,z)\,dz\,dw\leq\mathcal{A}_{{d,\alpha}}2^{d+\alpha}\delta^{-d-\alpha},
		\end{align*}
		and, applying the Lebesgue dominated convergence theorem, we get
		\begin{align*}
		\int_0^t\int_D\int_{D\setminus A} \phD(s,x,w)\hat{\nu}(w,z) p_D(t-s,z,y_n)\,dz\,dw\,ds\\
		\to \int_0^t\int_D\int_{D\setminus A}\phD(s,x,w)\hat{\nu}(w,z) p_D(t-s,z,y)\,dz\,dw\,ds,
		\end{align*}
as $n\to \infty$. Thus, $D\ni y\mapsto p_D(t,x,\cdot)$ is continuous.
	\end{proof}
	\noindent
	The proof of the following theorem is based on the proof of \cite[Theorem 2.4]{MR1329992}.
	\begin{theorem}[Chapman-Kolmogorov equation]	For all $t>s>0, x,y\in D$,
		\begin{align*}
		\phD(t,x,y)=\int_D \phD(s,x,z)\phD(t-s,z,y)\,dz.
		\end{align*}
	\end{theorem}
	\begin{proof}
		Fix $t>s>0, x\in D$. Let $A\subset D$. By the Markov property,
		\begin{align*}
		\int_A \phD(t,x,y)\,dy=\PP_x\left(t<\sigma_D,X_t\in A\right)=\E_x\left[s<\sigma_D;\PP_{X_s}\left(t-s<\sigma_D, X_{t-s}\in A \right) \right]\\
		=\E_x\left[s<\sigma_D;\int_A \phD(t-s,X_{t-s},y)\,dy\right]=\int_A\int_D \phD(s,x,z)\phD(t-s,z,y)\,dz\,dy.
		\end{align*}
		Thus $\phD(t,x,y) =\int_Ddz\ \phD(s,x,z)\phD(t-s,z,y)$ for almost every $y\in D$. The identity actually holds for all $y\in D$. Indeed, the function on the left-hand side is continuous in $y$. Since $\phD(t-s,z,y)\leq C(t-s)^{-d/\alpha}$, by the Lebesgue dominated convergence theorem the right-hand side is also continuous. The proof is complete.
	\end{proof}

\begin{lemma}\label{lem:phk}
	$\phD(t,x,y)>0$ for $x,y\in D$, $t>0$. 
\end{lemma}
\begin{proof}
	There is $n\geq2$ and a sequence of balls $\{B(x_{i},r_{i})\}_{i=1}^{n}$ in $D$ such that $x_{1}=x$, $x_{n}=y$ and $B(x_{i},r_{i})\cap B(x_{i+1},r_{i+1})\neq\emptyset$ for $i=1\dots n$.  
	By the Chapman-Kolmogorov equation and domain monotonicity of the heat kernel
	\begin{align*}
	\phD(t,x,y)&=\int_{D}
	\cdots\int_{D}\phD(t/n,x,z_{1})
	\cdots \phD(t/n,z_{n-1},y)\,dz_{1}
	\cdots\,dz_{n-1}\\
	&\geq\int_{B_{1}}
	\cdots\int_{B_{n-1}}\hat{p}_{B_1}(t/n,x,z_{1})
	\cdots \hat{p}_{B_n}(t/n,z_{n-1},y)\,dz_{1}
	\cdots\,dz_{n-1}.
	\end{align*}
Since the balls are convex sets{, we have $\hat{p}_{B_i} = p_{B_i}$ for $i=1\dots n$, because $\sigma_{B_i}=\tau_{B_i}$; see the Introduction, Definition~\ref{def_hunt}, and \eqref{eq:Hfk}. Therefore,} the right-hand side equals
	$$\int_{B_{1}}
	\cdots\int_{B_{n-1}}p_{B_1}(t/n,x,z_{1})
	\cdots p_{B_n}(t/n,z_{n-1},y)\,dz_{1}
	\cdots\,dz_{n-1},$$
	which is bigger than $0$ by the strict positivity of the heat kernel of killed process; see \cite{MR2677618} or \cite{MR2722789}.
\end{proof}

\subsection{Symmetry}
The goal of this section is to prove the following result.
\begin{theorem}\label{thm:symmetry} For all $x,y\in\Rd,\ t>0$, we have $\phD(t,x,y)=\phD(t,y,x)$.
\end{theorem}
The proof is given at the end of the section, after several auxiliary results.	
Fix $x,y\in\Rd,\ t>0$. We will construct, in Lemma~\ref{lem:btxy}, the so-called bridge between $x$ and $y$ in time $t$. We begin by defining the finite-dimensional distributions of the bridge.
\begin{definition}
	For $n\in \N,\ s_1,\dots,s_n\in (0,t),\ s_1<\dots<s_n$, we define measure $\pi_{s_1,\dots,s_n}$ on $(\Rd)^n$ by
	$$
	\pi_{s_1,\dots,s_n}(A_1\times\dots\times A_n) = \int_{A_n} \dots \int_{A_1} \frac{p(s_1,x,z_1)p(s_2-s_1,z_1,z_2)\dots p(t-s_n,z_n,y)}{p(t,x,y)} dz_1\dots dz_n
	$$
	for $A_1,\dots,A_n\subset \Rd$. 
	We also define
	$$
	\pi_{0,s_1,\dots,s_n}(A_0\times A_1\times\dots\times A_n)=\delta_x(A_0)\pi_{s_1,\dots,s_n}(A_1\times\dots\times A_n),
	$$
	$$
	\pi_{s_1,\dots,s_n,t}(A_1\times\dots\times A_n\times A_{n+1})=\pi_{s_1,\dots,s_n}(A_1\times\dots\times A_n)\delta_y(A_{n+1}),
	$$
	$$
	\pi_{0,s_1,\dots,s_n,t}(A_0\times A_1\times\dots\times A_n\times A_{n+1})=\delta_x(A_0)\pi_{s_1,\dots,s_n}(A_1\times\dots\times A_n)\delta_y(A_{n+1}).
	$$
	As usual, we extend the definition to $0\leq s_1\leq s_n\leq t$ by skipping the repeated $s_i$'s, e.g.,
	$$
	\pi_{s,s}(A_1\times A_2)=\pi_s(A_1\cap A_2).
	$$
\end{definition}
\begin{lemma}\label{lemstochcont} There is a constant $C>0$ such that for all $s\in(0,t)$ and $\varepsilon>0$,
	\begin{align*}
	\pi_s(z:|z-y|\geq \varepsilon)\leq \frac{C(t-s)}{p(t,x,y)\varepsilon^{d+\alpha}},
	\end{align*}
	and for $0<s_1<s_2<t, \varepsilon>0$,
	\begin{align*}
	\pi_{s_1,s_2}((z_1,z_2)\in \Rd\times\Rd: |z_1-z_2|\geq\varepsilon) \leq \frac{C(s_2-s_1)}{p(t,x,y)\varepsilon^{d+\alpha}}.
	\end{align*}
	
\end{lemma}
\begin{proof}	By \eqref{pt:approx},
	$$
	p(s,z,w) \approx s^{-d/\alpha}\wedge\frac{s}{|z-w|^{d+\alpha}}.
	$$
	To prove the first inequality we 
	write
	\begin{align*}
	\pi_s(z:|z-y|\geq \varepsilon)&=\int_{B_\varepsilon(y)^c}\frac{p(s,x,z)p(t-s,z,y)}{p(t,x,y)}\,dz\\
	&\leq \frac{C(t-s)}{p(t,x,y)\varepsilon^{d+\alpha}}\int_{B_\varepsilon(y)^c}p(s,x,z)\,dz\\
	&\leq\frac{C(t-s)}{p(t,x,y)\varepsilon^{d+\alpha}}.
	\end{align*}
	To prove the second inequality 
	we write
	\begin{align*}
	&\pi_{s_1,s_2}((z_1,z_2): |z_1-z_2|\geq\varepsilon) = \int_{\Rd}\int_{B_\varepsilon(z_1)^c} \frac{p(s_1,x,z_1)p(s_2-s_1,z_1,z_2)p(t-s_2,z_2,y)}{p(t,x,y)}\,dz_2\,dz_1\\
	&\leq \frac{C(s_2-s_1)}{\varepsilon^{d+\alpha}}\int_{\Rd}\int_{B_\varepsilon(z_1)^c} \frac{p(s_1,x,z_1)p(t-s_2,z_2,y)}{p(t,x,y)}\,dz_2\,dz_1
	\leq \frac{C(s_2-s_1)}{p(t,x,y)\varepsilon^{d+\alpha}}.\qedhere
	\end{align*}
\end{proof}
\begin{lemma}\label{lem:btxy}
	There exists a probability measure $\PP_x^{t,y}$ on $\Omega$, whose finite-dimensional distributions are given by $\pi_{s_1,\dots,s_n}$, $0\leq s_1\leq\dots\leq s_n\leq t$.
\end{lemma}
\begin{proof}
	By the Chapman-Kolmogorov equation for $p$, $\pi_{s_1,\dots,s_n}$ are consistent probability measures and so we can use the Kolmogorov existence theorem. We will verify the remaining conditions of \cite[Theorem 13.6]{MR1700749}, namely
	\begin{equation}\label{condition1}	
	\pi_{s_1,s,s_2}((z_1,z,z_2):|z_1-z|\wedge|z_2-z|\geq\lambda) \leq \frac{1}{\lambda^{\beta}}(F(s_2)-F(s_1))^\gamma, 0\leq s_1\leq s \leq s_2\leq t,
	\end{equation}
	with $F$ a suitable nondecreasing, continuous function on $[0,t]$, $\lambda >0$, $\beta\geq 0$, $\gamma>1$, and
	\begin{equation}\label{condition2}
	\lim\limits_{h\downarrow 0} \pi_{s,s+h}((z_1,z_2):|z_1-z_2|\geq\varepsilon) = 0, 0\leq s<t.
	\end{equation}
	To verify (\ref{condition1}) we write
	\begin{align*}
	&\pi_{s_1,s,s_2}((z_1,z,z_2):|z_1-z|\wedge|z_2-z|\geq\lambda)\\
	&= \int_{\Rd} \int_{B_\lambda(z_1)^c} \int_{B_\lambda(z)^c}  \frac{p(s_1,x,z_1) p(s-s_1,z_1,z) p(s_2-s,z,z_2) p(t-s_2,z_2,y)}{p(t,x,y)}\,dz_2\,dz\,dz_1\\
	&\leq C\int_{\Rd} \int_{B_\lambda(z_1)^c} \int_{B_\lambda(z)^c} p(s_1,x,z_1)\frac{s-s_1}{|z_1-z|^{d+\alpha}} \frac{s_2-s}{{|z_2-z|^{d+\alpha}}} p(t-s_2,z_2,y) \,dz_2\,dz\,dz_1\\
	&\leq  C\frac{1}{{\lambda^{2d+2\alpha}}}(s_2-s_1)^2.
	\end{align*}
	Thus the first condition holds with $\beta = {2d+2\alpha}$, $\gamma = 2$, $F(x) = \sqrt{C}x$.
	\newline
	To see that (\ref{condition2}) is also satisfied, for $0<s<s+h<t$ we use Lemma \ref{lemstochcont}, to get
	\begin{align*}
	\pi_{s,s+h}((z_1,z_2):|z_1-z_2|\geq\varepsilon) \leq \frac{Ch}{p(t,x,y)\varepsilon^{d+\alpha}} \to 0
	\end{align*}
	as $h \to 0$.
\end{proof}
\noindent	
\begin{corollary}
	Under $\PP_x^{t,y}$, $X$ is stochastically continuous on $[0,t]$.
\end{corollary}
\begin{proof}
	First, we will prove the stochastic continuity on $[0,t)$. To this end fix $s\in[0,t)$, $\varepsilon>0$ and take $r\neq s$. Then by Lemma \ref{lemstochcont},
	$$
	\PP_x^{t,y}(|X_s-X_r|\geq\varepsilon)\leq \frac{C|s-r|}{p(t,x,y)\varepsilon^{d+\alpha}} \to 0
	$$
	as $r\to s$.
	To see stochastic continuity at $t$, by using Lemma \ref{lemstochcont}, for $s<t$ we get
	\begin{displaymath}
	\PP_x^{t,y}(|X_s-y|\geq\varepsilon)\leq \frac{C(t-s)}{p(t,x,y)\varepsilon^{d+\alpha}}\to 0\quad \mbox{ as $s\to t$.}\qedhere
	\end{displaymath}
	
\end{proof}
We denote $X_{0-} = X_0$.

\begin{corollary}\label{eqlim}
	For every $s\in [0,t]$, we have $\PP_x^{t,y}(X_{s-}=X_s)=1$.
\end{corollary}
\begin{proof}
	There is nothing to prove for $s=0$, so let $s\in(0,t]$. We have
	$$
	1-\PP_x^{t,y}(X_{s-}=X_s) = \PP_x^{t,y}\left(\exists_{k>0} |X_{s-}-X_s|>\frac{1}{k}\right) \leq \sum_{k>0} \PP_x^{t,y}\left(|X_{s-}-X_s|>\frac{1}{k}\right),
	$$
	so it suffices to show that $\PP_x^{t,y} (|X_{s-}-X_s|>\frac{1}{k}) = 0$ for any $k>0$. To this end fix $0<s_n\uparrow s$. Since $X$ has $\cadlag$ trajectories and is stochastically continuous, 
	$$
	\PP_x^{t,y} \left(|X_{s-}-X_s|>\frac{1}{k}\right) \leq \PP_x^{t,y} \left(|X_{s-} - X_{s_n}|+|X_{s_n} - X_s|>\frac{1}{k}\right)
	$$
	$$
	\leq \PP_x^{t,y} \left(|X_{s-} - X_{s_n}|>\frac{1}{2k}\right) + \PP_x^{t,y} \left(|X_{s_n} - X_s|>\frac{1}{2k}\right) \to 0
	$$
	as $n\to\infty$, which finishes the proof.
\end{proof}
\noindent		
We define the time-reversed process $X'$ by $X'_s=X_{(t-s)-}$ for $s\in [0,t)$, $X'_t=X_0$. Then $X'$ also has $c\grave{a}dl\grave{a}g$ trajectories.
We are ready to prove the following result.
\begin{lemma}\label{symdistr}
The distribution of $X'$ under $\PP_y^{t,x}$ equals that of $X$ under 
$\PP_x^{t,y}$.
\end{lemma}
\begin{proof}
	It is enough to compare the finite-dimensional distributions of $X$ and $X'$, since they uniquely define their laws.
	To this end, fix measurable sets $A_1,\dots,A_n\subset \Rd$ and $0<s_1<\dots<s_n<t$. Using Corollary \ref{eqlim}, the symmetry of $p$ and the fact that $(t-s_{i-1})-(t-s_i)=s_i-s_{i-1}$, for $i=n,n-1, \dots, 2$, we get
	\begin{align*}
	&\PP_y^{t,x}(X'_{s_1}\in A_1, \dots, X'_{s_n}\in A_n) = \PP_y^{t,x}(X_{(t-s_1)-}\in A_1, \dots, X_{(t-s_n)-}\in A_n)\\
	&= \PP_y^{t,x}(X_{(t-s_1)-}\in A_1, \dots, X_{(t-s_n)-}\in A_n, X_{(t-s_1)-}=X_{t-s_1}, \dots , X_{(t-s_n)-}=X_{t-s_n})\\
	&= \PP_y^{t,x}(X_{t-s_1}\in A_1, \dots, X_{t-s_n}\in A_n, X_{(t-s_1)-}=X_{t-s_1}, \dots , X_{(t-s_n)-}=X_{t-s_n}) \\
	&= \PP_y^{t,x}(X_{t-s_1}\in A_1, \dots, X_{t-s_n}\in A_n)
	= \PP_y^{t,x}(X_{t-s_n}\in A_n, \dots, X_{t-s_1}\in A_1) \\
	&= \int_{A_1}\dots \int_{A_n} \frac{p(t-s_n,y,z_n)p(s_n-s_{n-1},z_n,z_{n-1})\dots p(s_2-s_1,z_2,z_1)p(s_1,z_1,x)} {p(t,y,x)}\, dz_n\dots dz_1\\
	&= \int_{A_n}\dots \int_{A_1} \frac{p(s_1,x,z_1)p(s_2-s_1,z_1,z_2)\dots p(s_n-s_{n-1},z_{n-1},z_n)p(t-s_n,z_n,y)} {p(t,x,y)}\,dz_1\dots dz_n\\
	&= \PP_x^{t,y}(X_{s_1}\in A_1, \dots, X_{s_n}\in A_n).
	\end{align*}
\end{proof}

\begin{proposition}\label{symbridge} For all $x,y\in\Rd,\, t>0$, we have $\PP_x^{t,y}(\sigma_D\geq t)=\PP_y^{t,x}(\sigma_D\geq t)$.
\end{proposition}
\begin{proof}
	Clearly, $\{\sigma_D\geq t\} = \{[X_{s-},X_s]\subset D, 0<s<t\}$.
	By Lemma \ref{symdistr},
	\begin{align*}
	\PP_x^{t,y}(\sigma_D(X)\geq t) &= \PP_x^{t,y}([X_{s-},X_s]\subset D, 0<s<t)
	= \PP_x^{t,y}([X'_{s},X'_{s-}]\subset D, 0<s<t)\\
	&= \PP_x^{t,y}(\sigma_D(X')\geq t)
	= \PP_y^{t,x}(\sigma_D(X)\geq t).\qquad\qedhere
	\end{align*}
\end{proof}
\begin{lemma}\label{Huntbridge}For all $x,y\in\Rd,\, t>0$, we have  $\phD(t,x,y)=\PP_x^{t,y}(\sigma_D\geq t)p(t,x,y)$.
\end{lemma}
\begin{proof}
	Dividing both sides of Hunt formula by $p(t,x,y)$ we get
	\begin{displaymath}
	\frac{\phD(t,x,y)}{p(t,x,y)}= 1-\frac{\E_x\left[\sigma_D<t;p(t-\sigma_D,X_{\sigma_D},y)\right]}{p(t,x,y)}=1-\PP_x^{t,y}(\sigma_D<t)=\PP_x^{t,y}(\sigma_D\geq t).\;\qedhere
		\end{displaymath}
\end{proof}

\begin{corollary}\label{cor:comp}
$
0\leq \phD(t,x, y) \leq \pD(t,x, y)\leq p(t,x, y) 
$ for all $x,y\in \Rd$, $t>0$.
\end{corollary}
\begin{proof}
Since $\sigma_D\le \tau_D$, Lemma~\ref{Huntbridge} and its analogue for $\pD$ and $\tau_D$ yield the result. 
\end{proof}
\begin{proof}[Proof of Theorem~\ref{thm:symmetry}]
	Use Proposition \ref{symbridge}, Lemma \ref{Huntbridge} and symmetry of $p(t,\cdot,\cdot)$.
\end{proof}

\subsection{Incomparability of the heat kernels $\phD$ and $\pD$}\label{subsec:comp}

For the remainder of the paper, we assume that 
$D$ is nonempty open bounded $C^{1,1}$ set. 

Recall that $\phD\le \pD$. Sharp estimates of the Dirichlet heat kernel $\pD$ 
are well known \cite{MR2677618}, see also \cite{MR2722789}.
{If $D$ is convex, then $\phD=\pD$. If $D$ is nonconvex, then the kernels are  \textit{not comparable}.
Here is an intuitive explanation: The \textit{off-diagonal}
estimate for \( p_D(t, x, y) \) and (small) time $t>0$ is \( t |x - y|^{-d - \alpha} \), which we interpret by saying that  the main contribution to $\pD(t,x,y)$  comes from a single jump from near \( x \) to near \( y \). If the line segment \([x, y]\) is not contained in \( D \) then such jumps are forbidden so the shot-down process has to make
at least two (independent) large jumps to get from near \( x \) to near \( y \) up to time $t$, therefore $\phD(t,x,y)$ is much smaller than $\pD(t,x,y)$.}
\begin{theorem}\label{thm:inc}
	$\phD$ and $\pD$
	are comparable if and only if $D$ is convex.
\end{theorem}
\begin{proof}
	If $D$ is convex, then $\tau_D=\sigma_D$ and $\pD=\phD$, so there is nothing to prove.
	If $D$ is nonconvex, then we can find $x,y\in D$ and $Q\in D^c$ such that $Q\in [x,y]$, i.e., $Q=\lambda x+(1-\lambda)y$ for some $0\le \lambda\le 1$.
	If ${\rm dist}(Q,D)>0$, then $r=:{\rm dist}(Q,D)\wedge {\rm dist}(x,D^c)\wedge {\rm dist}(y,D^c)>0$, and
	we consider 
	the balls $B(x,r), B(y,r)\subset D$. For all the points $z\in B(x,r)$, $w\in B(y,r)$,  we have $\lambda z+(1-\lambda)w\in B(Q,r)$.  The latter is needed below---but before we proceed, we also consider the case of $Q\in \partial D$. 
	In this case we translate $Q$ slightly in the direction of the center of the exterior ball tangent to $\partial D$ at $Q$ 
	(see Definition~\ref{def:C11}), and we translate $x$ and $y$ by the same vector, which moves us to the first case.  
	To summarize---for nonconvex $C^{1,1}$ set $D$ there are balls $B(x,r), B(y,r)\subset D$ such that 
	for all $z\in B(x,r)$ and $w\in B(y,r)$, the line segment $[z,w]$ intersects the interior of $D^c$, in particular ${\rm dist}(B(x,r), B(y,r))>0$,  see {Figure}~\ref{fig:balls}.
	\begin{figure}
		\begin{center}
			\begin{tikzpicture}[scale=2]
			\draw [shift={(-11.95,8.4)},line width=1.2pt]  plot[domain=2.831889709047336:4.007894916142469,variable=\t]({1*1.3124404748406695*cos(\t r)+0*1.3124404748406695*sin(\t r)},{0*1.3124404748406695*cos(\t r)+1*1.3124404748406695*sin(\t r)});
			\draw [shift={(-12.05,8.4)},line width=1.2pt]  plot[domain=-0.8663022625526775:0.3097029445424578,variable=\t]({1*1.3124404748406684*cos(\t r)+0*1.3124404748406684*sin(\t r)},{0*1.3124404748406684*cos(\t r)+1*1.3124404748406684*sin(\t r)});
			\draw [shift={(-12.65,8.65)},line width=1.2pt]  plot[domain=0.6610431688506853:2.8753406044388665,variable=\t]({1*0.570087712549568*cos(\t r)+0*0.570087712549568*sin(\t r)},{0*0.570087712549568*cos(\t r)+1*0.570087712549568*sin(\t r)});
			\draw [shift={(-11.35,8.65)},line width=1.2pt]  plot[domain=0.2662520491509295:2.4805494847391056,variable=\t]({1*0.5700877125495686*cos(\t r)+0*0.5700877125495686*sin(\t r)},{0*0.5700877125495686*cos(\t r)+1*0.5700877125495686*sin(\t r)});
			\draw [shift={(-12.721132168535297,8.68577126967102)},line width=1.2pt]  plot[domain=-0.14821767022916887:0.5426028080631197,variable=\t]({1*0.6085379462666762*cos(\t r)+0*0.6085379462666762*sin(\t r)},{0*0.6085379462666762*cos(\t r)+1*0.6085379462666762*sin(\t r)});
			\draw [shift={(-11.291769672302514,8.687848474210886)},line width=1.2pt]  plot[domain=2.590802412473704:3.2892961130490246,variable=\t]({1*0.5964366194692994*cos(\t r)+0*0.5964366194692994*sin(\t r)},{0*0.5964366194692994*cos(\t r)+1*0.5964366194692994*sin(\t r)});
			\draw [shift={(-12,8.508196964730805)},line width=1.2pt]  plot[domain=4.087120120443383:5.337657840325996,variable=\t]({1*1.3667847352961513*cos(\t r)+0*1.3667847352961513*sin(\t r)},{0*1.3667847352961513*cos(\t r)+1*1.3667847352961513*sin(\t r)});
			\draw [line width=1.2pt] (-12.39996838597594,8.805243425377359) circle (0.102cm);
			\draw [line width=1.2pt] (-11.606001370593122,8.799405432617192) circle (0.102cm);
			\draw [shift={(-12.000002626550868,8.569245996038335)},line width=1.2pt]  plot[domain=-3.3653023323877065:0.25493143478218927,variable=\t]({1*0.12224130597605011*cos(\t r)+0*0.12224130597605011*sin(\t r)},{0*0.12224130597605011*cos(\t r)+1*0.12224130597605011*sin(\t r)});
			\begin{scriptsize}
			\draw [fill=black] (-12.39996838597594,8.805243425377359) circle (0.5pt);
			\draw[color=black] (-12.412209534460503,8.45619913218483) node {$B(x, r)$};
			\draw [fill=black] (-11.606001370593122,8.799405432617192) circle (0.5pt);
			\draw[color=black] (-11.560739392977234,8.453300886075097) node {$B(y, r)$};
			\end{scriptsize}
			\end{tikzpicture}
		\end{center}
		\caption{Situation in the proof of Theorem~\ref{thm:inc}}
      \label{fig:balls}
	\end{figure}
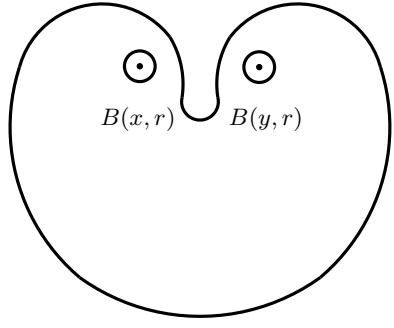
	\noindent
	Denote $B=B(x,r)$. By Lemma~\ref{lem:Hffssd} and {convexity of $B$},
	$$\phD(t,x, y) = \E_{x} [\tau_B<t, [X_{\tau_B-},X_{\tau_B}]\subset D; \phD(t-\tau_B,X_{\tau_B}, y)] + p_B(t,x, y).$$
	\noindent
	The second term above equals $0$, because $y \notin B$. The condition $[X_{\tau_B-}, X_{\tau_B}] \subset D$ implies that $X_{\tau_B} \notin B(y,r)$, i.e., $|X_{\tau_B} - y| \geq r$. Thus,  
	$$\phD(t,x, y) 
	\leq \PP_{x}(\tau_B<t) 
	\sup_{\substack{s<t \\ |z|> r}}p_s(z).$$
	If $0<s<t \leq 1$ and  $|z|> r$, then
	$$p_s(z)\approx s^{-d/\alpha} \land \frac{s}{|z|^{d+\alpha}} \leq cs \leq ct.$$
	Since $\PP_{x}(\tau_B<t)$ converges to 0 as $t \to 0$, we get
	$\phD(t,x, y) 
	= o(t)$.
	But for small $t>0$ we have $\pD(t,x,y) \approx t|x-y|^{-d-\alpha}$ \cite[Theorem~1.1]{MR2677618}.
\end{proof}

\section{The killing measures}\label{sec:lk_bound} 
Recall the killing intensities $\kappa_D$ and $\iota_D$, defined in the Introduction
\begin{definition}
The killing intensity for $D$ is $\kappa_D(x) = \int\limits_{D^c} \nu(y-x)\,dy, \enspace x \in D.$
\end{definition}

\noindent Given $x \in D$, recall the definition $D_x = \{y \in D : [x,y]\subset D \}$; we regard $D_x$ as the set of points that are ``visible in $D$ from $x$''---the set contains precisely those points where the process can jump from $x$ without being shot down. 
\begin{definition}
The shooting-down intensity for $D$ is $\iota_D(x) = \int\limits_{D_x^c} \nu(y-x)\,dy, \enspace x\in D.$
\end{definition}

Clearly, $\iota_D\ge \kappa_D>0$ and $\iota_D, \kappa_D$ are continuous on $D$. Let 
$$\delta_D(x) = {\rm dist}(x,D^c), \quad x\in \Rd.$$
We will estimate the difference $\iota_D- \kappa_D$ for $C^{1,1}$ open sets $D$ in terms of $\delta_D$.

Let $x \in D$, an open set in $\Rd$ that is $C^{1,1}$ at scale $r>0$. 
Let $Q \in \partial D$ be such that $\delta_D(x) = |x-Q|$. Let ${I(x)} = B(x', r)$ and ${O(x)} = B(x'', r)$ be the inner and outer balls tangent at $Q$, see Figure~\ref{fig:balls2}.
We have
$$
\kappa_D(x)\le \int \limits_{|y|>\delta_D(x)} \mathcal{A}_{{d,\alpha}}|y|^{-d-\alpha} dy=c \delta_D(x)^{-\alpha},
$$
and
$$
\kappa_D(x)\ge \int \limits_{{O(x)}} \mathcal{A}_{{d,\alpha}}|y|^{-d-\alpha}dy\ge c \delta_D(x)^{-\alpha} \quad \text{ if } \delta_D(x)\le r.
$$
Since $D^c\subset D_x^c  $, we have $\kappa_D\le \iota_D$, as noted before. 
Furthermore, $$D_x^c \setminus D^c \subset {I(x)}^c \setminus D^c = D \setminus {I(x)} \subset {O(x)}^c \setminus {I(x)} =\Rd \setminus ({O(x)} \cup {I(x)}),$$ hence
\begin{equation}\label{eq:lkbou}
    0 \leq \iota_D(x) - \kappa_D(x) =\!\!\! \int\limits_{D_x^c \setminus D^c} \!\!\!\nu(y-x)\,dy \leq\!\!\!\!\!\! \int\limits_{\Rd \setminus ({O(x)} \cup {I(x)})} \!\!\!\!\!\!\!\!\! \nu(y-x)\,dy.
\end{equation} 
\begin{theorem}\label{thm:intens-diff-bd}
\label{dif_bound}
If $D$ is a bounded $C^{1,1}$ open set, then
\begin{equation}\label{eq:lkboubs}
    0 \leq \iota_D(x) - \kappa_D(x) = o(\kappa_D(x))
 \quad \text{ as }    \delta_D(x)\to 0.
\end{equation} 
\end{theorem}
\begin{proof}
Let {$D$ be $C^{1,1}$ at a scale} $r>0$ 
{and} ${\rm dist}(x, \partial D) < r/4$. Let ${I(x)}$ and ${O(x)}$ be as in \eqref{eq:lkbou}.  
Clearly, $x \in {I(x)} = B(x', r)$. 

\smallskip

$\mathbf{Case \enspace 1: d=1}$. The set $D$ is a sum of open intervals of length at least $2r$ each and at a distance at least $2r$ from each other, therefore 

$$\iota_D(x) - \kappa_D(x) \leq 2\int_{2r}^{\infty} \mathcal{A}_{1,\alpha} s^{-1-\alpha}ds
 = cr^{-\alpha},$$
whereas $\kappa_D(x)\geq c\delta_D(x)^{-\alpha}$. This completes the proof of \eqref{eq:lkboubs} for $d=1$.

\smallskip

$\mathbf{Case \enspace 2: d=2}$. By possibly changing coordinates, we may assume that $Q=0 \in \partial D$, $x=(0, x_2)$, $x'= (0, r)$, and $x''= (0, -r)$. By our assumptions, we have: $0 < x_2 < r/4$.

 \begin{figure}[H]
   \begin{center}
     \begin{tikzpicture}[rotate=-90,scale=4.5]
       \draw[line width=1pt] (-0.3, 0) circle (0.3) node [above=20] {${I(x)}$};
       \draw[line width=1pt] (0.3, 0) circle (0.3) node [below=20] {${O(x)}$};
       \draw [line width=1pt] plot [smooth] coordinates {(-0.1, -1.5) (-0.15, -1.2) (0.08,-0.7) (0,0) (0.1,0.9) (-0.2, 1.4)} node [above left=20pt] {$D$};
       \fill (-0.3,0) circle (0.015) node [above right=-2pt] {$x'$};
       \fill (0.3,0) circle (0.015) node [above right=-2pt] {$x''$};
       \fill (-0.07,0) circle (0.015) node  [above right=-2pt]{\small $x$};
     \end{tikzpicture}
   \end{center}
   \caption{The point $x\in D$ and the balls ${I(x)}\subset D$ and ${O(x)}\subset D^c$}
   \label{fig:balls2}
 \end{figure}
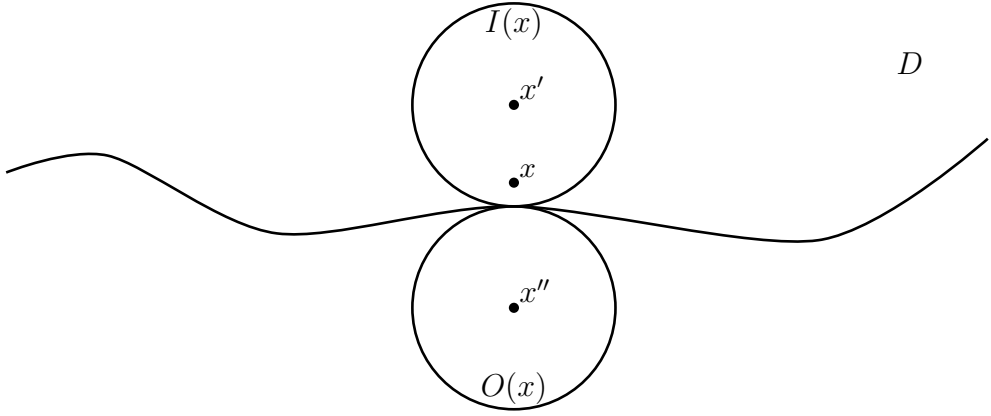
\noindent From \eqref{eq:lkbou}, we have
$$\iota_D(x) - \kappa_D(x) \leq \!\!\!\!\!\! \int\limits_{\mathbb{R}^2 \setminus  ({I(x)} \cup {O(x)})} \!\!\!\!\!\!   \nu(y-x)\,dy,$$
and it is clear that
$$\int\limits_{\mathbb{R} \times \mathbb{R}_- \setminus  {O(x)}}\!\!\!\!\!\!   \nu(y-x)\,dy \leq  \!\!\!\!\!\!\int\limits_{\mathbb{R} \times \mathbb{R}_+ \setminus  {I(x)}}\!\!\!\!\!\!   \nu(y-x)\,dy.$$
Thus,
$$\iota_D(x) - \kappa_D(x) \leq 2 \!\!\!\!\!\! \int\limits_{\mathbb{R} \times \mathbb{R}_+ \setminus  {I(x)}} \!\!\!\!\!\! \mathcal{A}_{2,\alpha} \nu(y-x)\,dy. $$
Considering $y=(y_1,y_2)$ in $\mathbb{R} \times \mathbb{R}_+ \setminus {I(x)}$ {and using polar coordinates}, yields 
$$|y-x|^2 = y_1^2 + (x_2-y_2)^2 = (\rho \cos \phi)^2 + (x_2 - \rho \sin \phi)^2 = \rho ^2 + x_2^2 - 2\rho x_2 \sin \phi .$$ 
In $\mathbb{R} \times \mathbb{R}_+ \setminus {I(x)}$ we have $2r\sin \phi\le \rho$, and since $x_2 \leq \frac{r}{4}$, we get $2 \rho x_2 \sin \phi \leq \frac{\rho^2}{4}$, so $|y-x|^2 \geq \frac{\rho ^2}{2} + x_2^2$.
Therefore,
\vspace*{-3mm}
\begin{align*}
\iota_D(x) - \kappa_D(x) &
\le 2 \!\!\!\!\!\! \int\limits_{\mathbb{R} \times \mathbb{R}_+ \setminus  {I(x)}} \!\!\!\!\!\! \mathcal{A}_{2,\alpha} |y-x|^{-2-\alpha}dy \leq 2 \int_{0}^{\pi}\!\!\! \int_{2r\sin\phi}^{\infty} \!\!\!\!\!\! \mathcal{A}_{2,\alpha}\left(\frac{1}{2}\rho^2 + x_2^2\right)^{\frac{-2-\alpha}{2}}\rho\,d\rho\,d\phi\\ 
& = 2\int_{0}^{\pi} \mathcal{A}_{2,\alpha}\frac{2}{\alpha}\left(\frac{1}{2}\cdot 4r^{2}\sin\phi^2 + x_2^2\right)^{-\frac{\alpha}{2}}d\phi \\
 &\leq\frac{8}{\alpha}\mathcal{A}_{2,\alpha}  \int_{0}^{\pi/2} \left(r^{2}\sin\phi^2 + x_2^2\right)^{-\frac{\alpha}{2}}d\phi \leq c\int_{0}^{1} \left(\phi^2 + x_2^2\right)^{-\frac{\alpha}{2}}d\phi\\
 &\leq c\int_{0}^{x_2} x_2^{-\alpha}d\phi + c\int_{x_2}^{1} \phi^{-\alpha}d\phi.
\end{align*}
This resolves the case $d=2$:
\begin{equation}\label{eq:or}
\iota_D(x) - \kappa_D(x) \leq c
\begin{cases}
\delta_D(x)^{-\alpha+1}, &\text{ $\alpha> 1$,}\\
\log\left(e+\frac{1}{\delta_D(x)}\right), &\text{ $\alpha = 1 $,}\\
1, &\text{ $\alpha < 1$.}
\end{cases}
\end{equation}

\smallskip

$\mathbf{Case \enspace 3: d=3}$. For $y = (y_1, y_2, ..., y_d) \in \Rd$, we let $\tilde y = (y_1, y_2, ...,y_{d-1})$, so that $y = (\tilde y, y_d)$. By possibly changing coordinates, we may assume that $Q=0 \in \partial D$, $x=(\tilde 0, x_d)$, $x'= (\tilde 0, r)$, and $x''= (\tilde 0, -r)$. Here $\tilde 0=(0, ..., 0) \in \mathbb{R}^{d-1}$. Thus, $0 < x_d < r/4$. Since
\vspace*{-2mm}
\begin{align*}
    \begin{split}
    \int\limits_{(\Rd \setminus ({O(x)} \cup {I(x)})) \cap \{y \in \Rd : |y| \geq r\}} 
    |x-y|^{-d-\alpha}dy 
&\leq 
\int\limits_{|z| > \frac{r\sqrt{3}}{2}} |z|^{-d-\alpha} dz = c r^{-\alpha},
    \end{split}
\end{align*}
we may restrict our attention to $(\Rd \setminus ({O(x)} \cup {I(x)})) \cap \{y \in \Rd : |y| < \frac{r\sqrt{3}}{2}\}$. 
\newline
\noindent
We are going to estimate $|x-y|$ in terms of $|y|$. Using the scalar product, we get 
\begin{enumerate}
\item $y \in {O(x)} \Longleftrightarrow |y|^2 < 2(y,x'')$,
\item $y \in {I(x)} \Longleftrightarrow |y|^2 < 2(y,x')$,
\end{enumerate}
hence $y \in \Rd \setminus ({I(x)} \cup {O(x)})$ if and only if $|y|^2 \geq 2|(y,x')|=2|y_d|r$, and then
\begin{align*}|x-y|^2 &= |\tilde y|^2 + |x_d-y_d|^2 = |y|^2 + x_d^2 - 2x_dy_d \\
&> |y|^2 + x_d^2 - 2|y_d|r/4 \ge 3|y|^2/4 + x_d^2 > |y|^2/2 + x_d^2.
\end{align*}
We get
\begin{align*}
\int\limits_{(\Rd \setminus ({O(x)} \cup {I(x)})) \cap \{y \in \Rd : |y| < \frac{r\sqrt{3}}{2}\}} \!\!\!\!\!\!\!\!\!\!\!\!\!\!\!\!\!\!\!\!\!\! |x-y|^{-d-\alpha}dy&\leq  
\!\!\!\!\!\!\!\!\!\!\!\!\!\!\!\!\!\!\! \int\limits_{(\Rd \setminus ({O(x)} \cup {I(x)})) \cap \{y \in \Rd : |y| < r\}} \!\!\!\!\!\!\!\!\!\!\!\!\!\!\!\!\!\!\!\!\!\!\!\!\! \left(\frac{|y|^2}{2} + x_d\right)^{\frac{-d-\alpha}{2}}dy\\
&= \!\!\!\!\! \int\limits_{\{y \in \Rd : |\tilde{y}| < r\}} \!\!\!\!\! \int_{0}^{|\tilde{y}|^2} \left(\frac{|\tilde{y}|^2}{2} +y_d + x_d\right)^{\frac{-d-\alpha}{2}}dy_d\, d\tilde{y} .
\end{align*}
Further, by cylindrical coordinates,
\begin{align*}
\int\limits_{\{y \in \Rd : |\tilde{y}| < r\}} \!\!\!\!\!\!\! \int_{0}^{|\tilde{y}|^2} \left(\frac{|\tilde{y}|^2}{2} +y_d + x_d\right)^{\frac{-d-\alpha}{2}}dy_d\, d\tilde{y}  &= c \int_{0}^{r}\int_{0}^{\rho^2} \left(\rho^2+y_d^2+x_d^2\right)^{\frac{-d-\alpha}{2}} dy_{d}\,d\rho\\
&\leq c \int_0^r\rho^{d-2}\rho^2\left(\rho^2+x_d^2\right)^{\frac{-d-\alpha}{2}} d\rho \\
&= \int_0^{x_d} \rho^{d}\left(\rho^2+x_d^2\right)^{\frac{-d-\alpha}{2}} d\rho + \int_{x_d}^r \rho^{d}\left(\rho^2+x_d^2\right)^{\frac{-d-\alpha}{2}} d\rho.
\end{align*}
For $\alpha \neq 1$, we estimate the first integral with $\int_0^{x_d} \rho^{d}(x_d^2)^{\frac{-d-\alpha}{2}} d\rho = cx_d^{-\alpha-1}$ and the second integral with 
$\int_{x_d}^r \rho^{d}(\rho^2)^{\frac{-d-\alpha}{2}} d\rho \leq cx_d^{-\alpha-1}$. Similarly, for $\alpha = 1$, the first integral may be bounded by $\int_0^{x_d} \rho^{d}(x_d^2)^{\frac{-d-\alpha}{2}} d\rho \leq c,$ and the second integral by $\int_{x_d}^r \rho^{d}(\rho^2)^{\frac{-d-\alpha}{2}} d\rho=c(\log r-\log x_d)$. Therefore, again,

\begin{align}\label{eq:or2}
\iota_D(x) - \kappa_D(x) \leq c
\begin{cases}
\delta_D(x)^{-\alpha+1}, &\text{ $\alpha> 1$,}\\
\log \left(e+\frac{1}{\delta_D(x)}\right), &\text{ $\alpha = 1 $,}\\
1, &\text{ $\alpha < 1$.}
\end{cases}
\end{align}\qedhere
\end{proof}

\section{Quadratic forms}\label{sec:quad-forms}

From now on we assume that $D$ is a bounded $C^{1,1}$ set in $\Rd$.
As usual, we let
$P^{D}_{t}f(x)=\int_{\Rd} \pD(t,x,y)f(y)\,dy$ and $\hat P^{D}_{t}f(x)=\int_{\Rd} \phD(t,x,y)f(y)\,dy$, $t>0$, $x\in \Rd$.
It is well known that $\{P^{D}_{t}\}_{t>0}$ is a strongly continuous contraction semigroup on $L^{2}(D)$.
Let $f\in L^2(D)$, that is $f\in L^2(\Rd)$ and $f=0$ on $D^c$. For $t>0$, we define as usual,
\begin{equation}\label{eq:diff}
\mathcal{E}^{D}_{t}[f]=\frac{1}{t}\langle f-P^D_{t}f,f \rangle
=\frac{1}{t}\langle f,f \rangle-\frac{1}{t}\langle P^D_{t/2}f,P^D_{t/2}f \rangle.
\end{equation}
Clearly, $\mathcal{E}^D_t[f]\ge 0$.
Let $E_\lambda$ be the spectral family of projections corresponding to (the generator of) $P^D_t$ \cite{MR2778606}. Then
$$\mathcal{E}^{D}_{t}[f]=\int_{[0,\infty)}\frac1t (1-e^{-\lambda t})\,d[E_\lambda f,f]= \int_{[0,\infty)}\int_0^\lambda e^{-\mu t}d\mu\, d[E_\lambda f,f] 
$$ is finite and nonincreasing in $t$, as is well known, see \cite[Lemma~1.3.4]{MR2778606}.   We define
$$\mathcal{E}^{D}[f]=\sup_{t>0}\mathcal{E}^{D}_{t}[f]=\lim_{t\rightarrow 0^{+}}\mathcal{E}^{D}_{t}[f],$$
the Dirichlet form of the killed process.
It is well known that 
$\mathcal{E}^{D}[f]=\mathcal{E}^{\Rd}[f]$, {for $f\in L^2(D)$}.
Since $p(t,x,y)=p(t,y,x)$, $\int p(t,x,y)\,dy=1$, $p(t,x,y)/t\leq c\nu(y-x)$ by \eqref{pt:approx}, {and $p(t,x,y)/t\to \nu(x,y)$ by \cite[(2.10)]{MR3613319},}  we get
\begin{align}
\nonumber
\mathcal{E}^{D}[f] &=\mathcal{E}^{\Rd}[f] 
=\lim_{t\to 0^+}\frac1t \int_{\Rd}\int_{\Rd} p(t,x,y) f(x)[f(x)-f(y)]\,dy\,dx\\
\nonumber
&=\lim_{t\to 0^+} \frac12 \int_{\Rd}\int_{\Rd} [f(y)-f(x)]^2p(t,x,y)/t\,dy\,dx\\
&=\frac12 \int_{\Rd}\int_{\Rd} [f(y)-f(x)]^2\nu(y-x)\,dy\,dx.
\label{e.fDf}
\end{align}
Since $f=0$ on $D^c$, we get the following Hardy-type inequality
\begin{align}
\mathcal{E}^{D}[f]
&\ge 2\frac12 \int_D\int_{D^c}f(x)^2\nu(y-x)\,dy\,dx
= \int_D f(x)^2\kappa_D(x)\,dx, \quad f\in L^2(D).\label{eq:Hi}
\end{align}
\begin{lemma}\label{lem:sc}
$\{\hat P^{D}_{t}\}_{t>0}$ is a strongly continuous contraction semigroup on $L^{2}(D)$.
\end{lemma}
\begin{proof}
The contractivity is trivial: by  Corollary~\ref{cor:comp} and Jensen's inequality, for $f\ge 0$,
\begin{align*}
\int_{\mathbb{R}^{d}}\left(\int_{\mathbb{R}^{d}}\phD(t,x,y)f(y)\,dy\right)^{2}dx&\leq \int_{\mathbb{R}^{d}}\left(\int_{\mathbb{R}^{d}}p(t,x,y)f(y)\,dy\right)^{2}dx\\
&\leq\int_{\mathbb{R}^{d}}\int_{\mathbb{R}^{d}}p(t,x,y)f(y)^{2}\,dy\,dx=\|f\|^{2}_{L^{2}(D)}.
\end{align*}	
For general $f\in L^2(D)$, we write $f=f_+-f_-$, where $f_+=f\vee 0$ and $f_-=(-f)\vee 0$, and we note that $\|f\|^{2}_{L^{2}(D)}=\|f_+\|^{2}_{L^{2}(D)}+\|f_-\|^{2}_{L^{2}(D)}=\|\,|f|\,\|^{2}_{L^{2}(D)}$, while $\left|\int_{\mathbb{R}^{d}}\phD(t,x,y)f(y)\,dy\right|\leq \int_{\mathbb{R}^{d}}\phD(t,x,y)|f(y)|\,dy$. This extends the contractivity to arbitrary (signed) $f\in L^2(D)$. Such extensions will be used tacitly in what follows. The semigroup property follows from Chapman--Kolmogorov equations for $\phD$.

For $t>0$ and nonnegative $f,g\in L^{2}(D)$, by Corollary~\ref{fact:huntIke}, Theorem~\ref{thm:symmetry}, the inequality $ab\leq \frac{1}{2}\left(a^{2}+b^{2}\right)$, Theorem~\ref{thm:intens-diff-bd}, Hardy inequality \eqref{eq:Hi}, and strong continuity of the semigroup $P^D_t$, we obtain
\begin{align*}
&0\leq \langle P_{t}^{D}f-\hat{P}^{D}_{t}f,g\rangle_{L^{2}(D)}\\
&=\int_{\mathbb{R}^{d}}\int_{\mathbb{R}^{d}}\int_{0}^{t}\int_{D}\int_{D\setminus D_{w}}\phD(s,x,w)\nu(z-w)\pD(t-s,z,y)f(y)g(x)\,dz\,dw\,ds\,dx\,dy\\
&=\int_{0}^{t}\int_{D}\int_{D\setminus D_{w}}\hat{P}^{D}_{s}g(w)\nu(z-w)P^{D}_{t-s}f(z)\,dz\,dw\,ds\\
&\leq\frac{1}{2}\int_{0}^{t}\int_{D}\int_{D\setminus D_{w}}\left(\hat{P}^{D}_{s}g(w)^{2}+P^{D}_{t-s}f(z)^{2}\right)\nu(z-w)\,dz\,dw\,ds\\
&\leq c\int_{0}^{t}\int_{D} P^{D}_{s}g(w)^{2}\kappa_D(w)\,dw\,ds+c\int_{0}^{t}\int_{D}P^{D}_{t-s}f(z)^{2}\kappa_D(z)\,dz\,ds\\
&\leq c \int_0^t \mathcal{E}^{D}[P^D_s g]\,ds + c \int_0^t \mathcal{E}^{D}[P^D_s f]\,ds\\
&=c\left( \|f\|_{L^2(D)}^2-\|P^D_t f\|_{L^2(D)}^2+\|g\|_{L^2(D)}^2-\|P^D_t g\|_{L^2(D)}^2\right)\to 0\quad \mbox{as}\quad t\to 0. 
\end{align*}
It follows that $\hat{P}^D_t=P^D_t-(P_t^D-\hat P^D_t)$ is weakly,  hence strongly continuous on $L^2(D)$, see \cite[Theorem 1.6]{MR2229872}.
\end{proof}

By Section~\ref{subsec:comp},  the heat kernels of $X$ and $\hat{X}$ are not comparable in general.
In this section, however, we will show that their quadratic forms are comparable.

The Dirichlet form of the shot-down process is defined in the usual fashion: we let
\begin{equation}\label{eq:diffhat}
\hat{\mathcal{E}}^{D}_{t}[f]=\frac{1}{t}\langle f-\hat{P}^D_{t}f,f \rangle
=\frac{1}{t}\langle f,f \rangle-\frac{1}{t}\langle \hat{P}^D_{t/2}f,\hat{P}^D_{t/2}f \rangle
\end{equation}
and consider the (monotone) limit
$$\hat{\mathcal{E}}^{D}[f]=\lim_{t\rightarrow 0^{+}}\hat{\mathcal{E}}^{D}_{t}[f]=\sup_{t>0}\hat{\mathcal{E}}^{D}_{t}[f].$$
The domains of the considered  forms $\mathcal{E}^{D}$ and $\hat{\mathcal{E}}^{D}$ are, respectively,
\begin{align*}
\mathcal{D}(\mathcal{E}^{D})=\{f\in L^{2}(D):\mathcal{E}^{D}[f]<\infty\}, \qquad
\mathcal{D}(\hat{\mathcal{E}}^{D})=\{f\in L^{2}(D):\hat{\mathcal{E}}^{D}[f]<\infty\}.
\end{align*}
They are linear subspaces of $L^2(D)$ because of the Cauchy-Schwarz inequality, e.g.,
\begin{equation}\label{eq:bf}
\hat{\mathcal{E}}^{D}_{t}(f,g):=\frac{1}{t}\langle f-\hat{P}^D_{t}f,g \rangle
\leq \hat{\mathcal{E}}^{D}_{t}[f]^{1/2}\hat{\mathcal{E}}^{D}_{t}[g]^{1/2},\quad t>0,\quad f,g\in L^2(D). 
\end{equation}
Then, $\mathcal{D}(\hat{\mathcal{E}}^{D})\subset \mathcal{D}(\mathcal{E}^{D})$. 
Indeed, if $f\in \mathcal{D}(\hat{\mathcal{E}}^{D})$, then the following is bounded for $t>0$,
\begin{align*}
\hat{\mathcal{E}}^{D}_{t}[f]&=\frac{1}{t}\left(\langle f,f\rangle -\langle\hat{P}^D_{t}f,f \rangle\right)\\
&=\frac{1}{t}\big(
\langle f_+,f_+\rangle -\langle\hat{P}^D_{t}f_+,f_+ \rangle
+
\langle f_-,f_-\rangle -\langle\hat{P}^D_{t}f_-,f_- \rangle
+
\langle\hat{P}^D_{t}f_+,f_- \rangle
+
\langle\hat{P}^D_{t}f_-,f_+ \rangle
\big),
\end{align*}
and so $\hat{\mathcal{E}}^{D}_{t}[f_+]$ is bounded for $t>0$. Thus $f_+\in \mathcal{D}(\hat{\mathcal{E}}^{D})$. Since $f_+\ge 0$,  by \eqref{eq:diff} and \eqref{eq:diffhat} we get
$\hat{\mathcal{E}}_t^{D}[f_+]\ge \mathcal{E}^{D}_t[f_+]$, so
$f_+\in \mathcal{D}(\mathcal{E}^{D})$. Similarly, $f_-\in \mathcal{D}(\mathcal{E}^{D})$, hence $f\in \mathcal{D}(\mathcal{E}^{D})$, as needed.
Here is the main result of this section.
\begin{theorem}\label{thm:Dirichlet-form}
Let $D\subset \mathbb{R}^{d}$ be an open bounded $C^{1,1}$ set, $\alpha\in(0,2)$, and $\nu$ as in \eqref{eq:alpha-stable-levy-density}. For $f \in \mathcal{D}(\mathcal{E}^{D})$,
\begin{equation}\label{eq:rbf}
\hat{\mathcal{E}}^{D}[f]=\mathcal{E}^{D}[f]+\int_{D}\int_{D\setminus D_{w}}f(w)\nu(z-w)f(z)\,dz\,dw.
\end{equation}
Furthermore, $\int_{D}\int_{D\setminus D_{w}}|f(w)\nu(z-w)f(z)|\,dz\,dw\leq
\mathcal{E}^{D}[f]<\infty$ and $\mathcal{D}(\hat{\mathcal{E}}^{D})=\mathcal{D}(\mathcal{E}^{D})$.
\end{theorem}

\begin{proof}
Let $f\in\mathcal{D}(\mathcal{E}^D)$ and $t>0$. We have $\hat{\mathcal{E}}^{D}_{t}[f]=\mathcal{E}^{D}_{t}[f]+\frac{1}{t}\langle P_{t}^{{D}}f-\hat{P}_{t}^{{D}}f,f\rangle$.
{By Corollary~\ref{fact:huntIke},}
$$
\langle P_{t}^{D}f-\hat{P}_{t}^{D}f,f\rangle=\int_{\mathbb{R}^{d}}\int_{\mathbb{R}^{d}}\left(\pD(t,x,y)-\phD(t,x,y)\right)f(y)f(x)\,dy\,dx:=I_{1}-I_{2},
$$
where
\begin{align*}
I_{1}&=\int_{\mathbb{R}^{d}}\int_{\mathbb{R}^{d}}\int_{0}^{t}\int_{D}\int_{D\setminus D_{w}}\pD(s,x,w)\nu(z-w)p_{D}(t-s,z,y)f(y)f(x)\,dz\,dw\,ds\,dx\,dy\\
&=\int_{0}^{t}\int_{D}\int_{D\setminus D_w}\nu(z-w) P_{s}^{D}f(w)P_{t-s}^{D}f(z)\,dz\,dw\,ds
\end{align*}
and
\begin{align*}
 I_{2}&=\int_{\mathbb{R}^{d}}\int_{\mathbb{R}^{d}}\int_{0}^{t}
\int_{D}\int_{D\setminus D_{w}}\int_{0}^{s}\int_{D}\int_{D\setminus D_{w'}}\phD(s',x,w')\nu(z'-w')\pD(s-s',z',w)\nu(z-w)\\
&\quad \times p_D(t-s,z,y)f(y)f(x)\,dz'\,dw'\,ds'\,dz\,dw\,ds\,dx\,dy.
\end{align*}
Since $\phD(t,u,v)$ and $\pD(t,u,v)$ are symmetric in the space variables,
\begin{align*}
I_{2}&=\int_{0}^{t}\int_{D}\int_{D\setminus D_{w}}\int_{0}^{s}\int_{D}\int_{D\setminus D_{w'}}\nu(z'-w')\pD(s-s',z',w)\nu(z-w)\\
&\quad \times \hat{P}^{D}_{s'}f(w')P^{D}_{t-s}f(z)\,dz'\,dw'\,ds'\,dz\,dw\,ds.
\end{align*}
\noindent Then,
\begin{align*}
2|I_{2}|&\leq\int_{0}^{t}\int_{D}\int_{D\setminus D_{w}}\int_{0}^{s}\int_{D}\int_{D\setminus D_{w'}}\nu(z'-w')\pD(s-s',z',w)\nu(z-w)\\
&\quad \times\left(P^{D}_{s'}f(w')^{2}+P^{D}_{t-s}f(z)^{2}\right)\,dz'\,dw'\,ds'\,dz\,dw\,ds=K_1+K_2=2K_1.
\end{align*}
To justify the latter equality, namely $K_1=K_2$, we change variables $t-s=u'$, $t-s'=u$, use the symmetry of $\nu$,  the symmetry of $p_D$ in the space variables and the equivalence: 
$z\in D\setminus D_{w}$ if and only if $w\in D\setminus D_{z}$. 

Since $|P^{D}_{t}f|\leq P^{D}_{t}|f|$, $|\hat{P}^{D}_{t}f|\leq \hat{P}^{D}_{t}|f|$ and $|f|\in\mathcal{D}(\mathcal{E}^D)$ by \eqref{e.fDf}, to show that $K_1$ is small, we may and do assume that $f\geq 0$. 
Denote $F_D(w)=\kappa_D(w)-\iota_D(w)$ and consider the following factor of the above integral
$$
L=L(s-s',z'):=\int_{D}\int_{D\setminus D_w}p_D(s-s',z',w)\nu(z-w)\,dz\,dw=\int_D p_D(s-s',z',w)F_D(w)\,dw.
$$
It is well known \cite[Theorem 1.1]{MR2677618} that for any $T>0$ and all $t\leq T,x,y\in D$,
$$
p_D(t,x,y)\approx\left(1\wedge\frac{\delta_D(x)^{\alpha/2}}{\sqrt{t}}\right)\left(1\wedge\frac{\delta_D(y)^{\alpha/2}}{\sqrt{t}}\right)\left(t^{-d/\alpha}\wedge\frac{t}{|x-y|^{d+\alpha}|}\right)
$$
with comparability constant depending only on $\alpha,\, D$ and $T$. Thus, for small $t$, $L/\left(1\wedge \frac{\delta_D(z')^{\alpha/2}}{\sqrt{s-s'}}\right)$ is comparable to 
{\allowdisplaybreaks
\begin{align*}
I_3&:=\int_D\left(1\wedge \frac{\delta_D(w)^{\alpha/2}}{\sqrt{s-s'}}\right)\left((s-s')^{-d/\alpha}\wedge\frac{s-s'}{|z'-w|^{d+\alpha}}\right)F_D(w)\,dw\\
&=\int_{\delta_D(w)>(s-s')^{1/\alpha}}\left((s-s')^{-d/\alpha}\wedge\frac{s-s'}{|z'-w|^{d+\alpha}}\right)F_D(w)\,dw\\
&\quad +\int_{\delta_D(w)\leq(s-s')^{1/\alpha}}\frac{\delta_D(w)^{\alpha/2}}{\sqrt{s-s'}}\left((s-s')^{-d/\alpha}\wedge\frac{s-s'}{|z'-w|^{d+\alpha}}\right)F_D(w)\,dw\\
&= \int_{\delta_D(w)>(s-s')^{1/\alpha},\, |z'-w|>(s-s')^{1/\alpha}}\frac{s-s'}{|z'-w|^{d+\alpha}}F_D(w)\,dw\\
&\quad +\int_{\delta_D(w)>(s-s')^{1/\alpha},\, |z'-w|\leq(s-s')^{1/\alpha}}(s-s')^{-d/\alpha}F_D(w)\,dw\\
&\quad +\int_{\delta_D(w)\leq(s-s')^{1/\alpha},\,|z'-w|>(s-s')^{1/\alpha}}\frac{\delta_D(w)^{\alpha/2}}{\sqrt{s-s'}}\frac{s-s'}{|z'-w|^{d+\alpha}}F_D(w)\,dw\\
&\quad +\int_{\delta_D(w)\leq(s-s')^{1/\alpha},\,|z'-w|\leq(s-s')^{1/\alpha}}\frac{\delta_D(w)^{\alpha/2}}{\sqrt{s-s'}}(s-s')^{-d/\alpha}F_D(w)\,dw.
\end{align*}
}
Let  $\alpha>1$. Then, by \eqref{eq:or2}, $F_D(w)\lesssim\delta_D(w)^{1-\alpha}$, hence, $I_3$ does not exceed a multiple of 
\begin{align*}
 &\int_{\delta_D(w)>(s-s')^{1/\alpha},\, |z'-w|>(s-s')^{1/\alpha}}\frac{(s-s')^{1/\alpha}}{|z'-w|^{d+\alpha}}\,dw\\ 
&+\int_{\delta_D(w)>(s-s')^{1/\alpha},\, |z'-w|\leq(s-s')^{1/\alpha}}(s-s')^{1/\alpha-1-d/\alpha}\,dw\\
&+\int_{\delta_D(w)\leq(s-s')^{1/\alpha},\,|z'-w|>(s-s')^{1/\alpha}}\frac{(s-s')^{1/\alpha}}{|z'-w|^{d+\alpha}}\,dw\\
&+\int_{\delta_D(w)\leq(s-s')^{1/\alpha},\,|z'-w|\leq(s-s')^{1/\alpha}}(s-s')^{1/\alpha-1-d/\alpha}\,dw \lesssim (s-s')^{1/\alpha-1}.
\end{align*}
If $\alpha=1$, then by \eqref{eq:or2}, we similarly obtain that $K_1\lesssim \log\left(e+1/(s-s')\right)$ and if $\alpha<1$, then $I_3$ is bounded, because $F_D$ is bounded and $p_D$ is a subprobability density.

Summarizing,
\begin{equation}\label{eq:L}
I_3\leq cg(s-s'),
\end{equation}
where, for $s>0$,
$$
g(s):=\begin{cases}
s^{1/\alpha-1},\,\alpha>1,\\
\log\left(e+\frac{1}{s}\right),\,\alpha=1,\\
1,\,\alpha<1.\\
\end{cases}
$$
By \eqref{eq:or2}, Hardy inequality \eqref{eq:Hi} and \eqref{eq:L} we have
\begin{align*}
 K_1&\lesssim\int_{0}^{t}\int_{0}^{s}\int_D\int_{D\setminus D_{w'}}g(s-s')\left(1\wedge \frac{\delta_D(z')^{\alpha/2}}{\sqrt{s-s'}}\right)\nu(z'-w')P_{s'}^{D}f(w')^2\,dz'\,dw'\,ds'\,ds\\
 &\leq \int_{0}^{t}\int_{0}^{s}\int_D\int_{D\setminus D_{w'}}g(s-s')\nu(z'-w')P_{s'}^{D}f(w')^2\,dz'\,dw'\,ds'\,ds\\
 &\lesssim \int_{0}^{t}\int_{0}^{s}\int_D g(s-s')\kappa_D(w')P_{s'}^{D}f(w')^2\,dw'\,ds'\,ds\\
 &\leq\int_{0}^{t}\int_{0}^{s}g(s-s')\mathcal{E}^{D}[P^{D}_{s'}f]\,ds'\,ds.
\end{align*}
We easily see that the function $G(s):=\int_{0}^{s}g(s')\,ds'$ is finite, nonnegative and bounded for small $s$, with $G(0)=0$, hence, by the monotonicity of $s\mapsto\mathcal{E}^{D}[P^{D}_{s}f]$, we get
\begin{align*}
K_1&\leq\int_{0}^{t}G(s)\sup_{s'\in[0,s]}\mathcal{E}^{D}[P^{D}_{s'}f]\,ds=\mathcal{E}^D[f]\int_{0}^{t}G(s)\,ds=o(t), \text{\,as\, } t\rightarrow 0^+.   
\end{align*}
We conclude that $\frac{1}{t}I_2\rightarrow 0$, as $t\rightarrow 0^+$.

\noindent
We now let $f\in\mathcal{D}(\mathcal{E}^D)$ be arbitrary, that is, not necessarily nonnegative, and  focus on the integral 
\begin{align*}
I_{1}=\int_{0}^{t}\int_{D}\int_{D\setminus D_{w}}P_{s}^{D}f(w)\nu(z-w)P_{t-s}^{D}f(z)\,dz\,dw\,ds.
\end{align*}
It is finite for nonnegative, hence arbitrary $f\in\mathcal{D}(\mathcal{E}^D)$. Indeed,
by Theorem~\ref{thm:intens-diff-bd}, Cauchy--Schwarz inequality, Hardy inequality \eqref{eq:Hi} and spectral theorem,
{\allowdisplaybreaks
\begin{align*}
I_1&\leq \sqrt{\int_{0}^{t}\int_{D}\int_{D\setminus D_{w}}P_{s}^{D}f(w)^{2}\nu(z-w)\,dz\,dw\,ds}\sqrt{\int_{0}^{t}\int_{D}\int_{D\setminus D_{w}}P_{t-s}^{D}f(z)^{2}\nu(z-w)\,dz\,dw\,ds}\\
&\lesssim \sqrt{\int_{0}^{t}\int_{D}P_{s}^{D}f(w)^{2}\kappa_D(w)\,dw\,ds}\sqrt{\int_{0}^{t}\int_{D}P_{t-s}^{D}f(z)^{2}\kappa_D(z)\,dz\,ds}\\
&\leq\int_{0}^{t}\mathcal{E}^{D}[P^{D}_{s}f]\,ds=\|f\|^{2}_{L^{2}(D)}-\|P_{t}^{D}f\|^{2}_{L^{2}(D)}<\infty.
\end{align*} 
}
for nonnegative, hence arbitrary $f\in\mathcal{D}(\mathcal{E}^D)$.

Moreover, $\{P^{D}_{t}\}_{t>0}$ is a 
semigroup of contractions, 
so by \eqref{eq:Hi},
{\allowdisplaybreaks
\begin{align*}
&\left|\frac{1}{t}I_{1}-\int_{D}\int_{D\setminus D_{w}}f(w)\nu(z-w)f(z)\,dz\,dw\right|\\
&\leq\frac{1}{t}\int_{0}^{t}\int_{D}\int_{D\setminus D_{w}}\nu(z-w)\left|P_{s}^{D}f(w)P_{t-s}^{D}f(z)-f(w)f(z)\right|\,dz\,dw\,ds\\
&\leq\frac{1}{t}\int_{0}^{t}\int_{D}\int_{D\setminus D_{w}}\nu(z-w)\left|P_{s}^{D}f(w)-f(w)\right||P_{t-s}^{D}f(z)|\,dz\,dw\,ds\\
&\quad +\frac{1}{t}\int_{0}^{t}\int_{D}\int_{D\setminus D_{w}}\nu(z-w)\left|P_{t-s}^{D}f(z)-f(z)\right||f(w)|\,dz\,dw\,ds\\
&\leq\frac{1}{t}\int_{0}^{t}\sqrt{\int_{D}\int_{D\setminus D_{w}} \hspace*{-3ex} \nu(z-w)\left|P_{s}^{D}f(w)-f(w)\right|^{2}\,dz\,dw}\sqrt{\int_{D}\int_{D\setminus D_{w}}\hspace*{-3ex}\nu(z-w)|P_{t-s}^{D}f(z)|^{2}\,dz\,dw}\,ds\\
&\quad +\frac{1}{t}\int_{0}^{t}\sqrt{\int_{D}\int_{D\setminus D_{w}}\hspace*{-3ex}\nu(z-w)\left|P_{t-s}^{D}f(z)-f(z)\right|^{2}\,dz\,dw}\sqrt{\int_{D}\int_{D\setminus D_{w}}\hspace*{-3ex}\nu(z-w)|f(w)|^{2}\,dz\,dw}\,ds\\
&\lesssim\frac{1}{t}\int_{0}^{t}\sqrt{\mathcal{E}^{D}[P^{D}_{s}f-f]\mathcal{E}^{D}[P_{t-s}f]}\,ds+\frac{1}{t}\int_{0}^{t}\sqrt{\mathcal{E}^{D}[P_{t-s}^{D}f-f]\mathcal{E}^{D}[f]}\,ds\rightarrow 0,
\end{align*}
}
as $t\rightarrow 0^+$, because the functions $s\mapsto\mathcal{E}^{D}[P^{D}_{s}f-f]\mathcal{E}^{D}[P_{t-s}f]$ and $\,s\mapsto\mathcal{E}^{D}[P_{s}f-f]$ are continuous and vanish at $s=0$.

This proves \eqref{eq:rbf}. We also note that by Theorem~\ref{thm:intens-diff-bd}
and the Hardy inequality \eqref{eq:Hi},
\begin{align}\label{eq:sk}
&\int_{D}\int_{D\setminus D_{w}}|f(w)\nu(z-w)f(z)|\,dz\,dw
\leq 
\frac12 \int_{D}\int_{D\setminus D_{w}} (f(w)^2+f(z)^2)\nu(z-w)\,dz\,dw
\\
&\leq 
c\int_{D}f(w)^2 \kappa_D(w)\,dw\leq c \mathcal{E}^{D}[f]<\infty.
\qedhere
\end{align}
\end{proof}

\medskip

Here and below, for $w \in D$ and $A \subset \R^d$,
we let $\nu(w,A) = \int_{A} \nu(z-w) dz$. We next propose and verify the following alternative to \eqref{eq:rbf}:
\begin{equation}\label{eq:rDfkp}
\hat{\mathcal{E}}^{D}[f]=\frac{1}{2}\iint\limits_{[z,w]\subset D}
(f(w)-f(z))^{2}\nu(z-w)\,dw\,dz+\int_{D}f^{2}(w)\nu(w,D_{w}^{c})\,dw.
\end{equation}

The first term on the right-hand side of \eqref{eq:rDfkp} can be thought of as coming from jumps, and the second is due to shooting down. 
To prove \eqref{eq:rDfkp} we argue as follows. Since $\Rd=D\cup D^c$ and $f=0$ on $D^{c}$,
$$\mathcal{E}^{D}[f]=\frac{1}{2}\int_{D}\int_{D}(f(w)-f(z))^{2}\nu(z-w)\,dw\,dz+\int_{D}f^{2}(w)\nu(w,D^{c})\,dw.$$
Therefore, by \eqref{eq:rbf}, 
{\allowdisplaybreaks
\begin{align*}
\hat{\mathcal{E}}^{D}[f]&=\frac{1}{2}\int_{D}\int_{D}(f(w)-f(z))^{2}\nu(z-w)\,dw\,dz+\int_{D}f^{2}(w)\nu(w,D^{c})\,dw\\
&\quad +\int_{D}\int_{D\setminus D_{w}}f(w)f(z)\nu(z-w)\,dz\,dw\\
&=\frac{1}{2}\int_{D}\int_{D_{w}}(f(w)-f(z))^{2}\nu(z-w)\,dw\,dz+\frac{1}{2}\int_{D}\int_{D\setminus D_{w}}(f(w)-f(z))^{2}\nu(z-w)\,dw\,dz\\
&\quad +\int_{D}f^{2}(w)\nu(w,D^{c})\,dw+\int_{D}\int_{D\setminus D_{w}}f(w)f(z)\nu(z-w)\,dz\,dw\\
&=\frac{1}{2}\iint\limits_{[z,w]\subset D}(f(w)-f(z))^{2}\nu(z-w)\,dw\,dz+\frac{1}{2}\int_{D}\int_{D\setminus D_{w}}\left(f^{2}(w)+f^{2}(z)\right)\nu(z-w)\,dw\,dz\\
&\quad +\int_{D}f^{2}(w)\nu(w,D^{c})\,dw,
\end{align*}
}which yields \eqref{eq:rDfkp}.
The steps above are justified by the absolute convergence of the integrals $\int_{D}\int_{D\setminus D_{w}}f(w)f(z)\nu(z-w)\,dw\,dz$ and $\int_{D}\int_{D\setminus D_{w}}f^{2}(w)\nu(z-w)\,dw\,dz$, see \eqref{eq:sk}.

\section{The Green function}\label{sec:Green}

\noindent
We recall the definition of the Green function for the killed process \cite[Section 2.3]{MR1329992}:
$$G_D(x,y) = \int \limits_{0}^{\infty} \pD(t,x,y)\,dt, \text{ } x,y \in \Rd.$$
\noindent
Analogously we define the Green function for the shot-down process.
\begin{definition}
$\hat{G}_D(x,y) = \int \limits_{0}^{\infty} \phD(t,x,y)\,dt$, $x,y \in \Rd$.
\end{definition}

{We next derive} a perturbation formula, Hunt formula, scaling, and positivity.  
\begin{lemmaf} \label{fact:Green_Pert}
$$\hat{G}_{D}(x,y)=G_{D}(x,y)-\int_{D}\int_{D\setminus D_{w}}\hat{G}_{D}(x,w)\nu(z-w)G_{D}(z,y)\,dz\,dw, \quad x, y \in \Rd.$$
\end{lemmaf}
\begin{proof}
By Corollary~\ref{fact:huntIke}, Fubini-Tonelli, and the above definitions,
\begin{align*}
    G_{D}(x,y) &=\hat{G}_{D}(x,y) -\int_{0}^{\infty}\int_{D}\int_{D\setminus D_{w}}\int_{0}^{t}\phD(s,x,w)\nu(z-w)\pD(t-s,z,y)\,ds\,dz\,dw\,dt\\
    &= \hat{G}_{D}(x,y) - \int_{D}\int_{D\setminus D_w}\int_{0}^{\infty}\phD(s,x,w)\nu(z-w)\,ds \int_{{s}}^{\infty}\,\pD(t-s,z,y)\,{dt}\,dz\,dw\\
    &= G_{D}(x,y)-\int_{D}\int_{D\setminus D_{w}}\hat{G}_{D}(x,w)\nu(z-w)G_{D}(z,y)\,dz\,dw.\qquad\qedhere
\end{align*}
\end{proof}

\begin{lemma}\label{lem:mnvG}
	For open $U\subset D$,
	\begin{align*}
		\hat{G}_D(x,y) = \hat{G}_U(x,y) + \E_x\left[[X_{\sigma_U-},X_{\sigma_U}]\subset D; \hat{G}_D(X_{\sigma_U},y)\right].
	\end{align*}
\end{lemma}
\begin{proof}
	By Lemma \ref{lem:Hffssd},
	\begin{align*}
		\hat{G}_D(x,y) &= \hat{G}_U(x,y) + \int_0^\infty \E_x\left[\sigma_U<t, [X_{\sigma_U-},X_{\sigma_U}]\subset D; \phD(t-\sigma_U,X_{\sigma_U},y)\right]\,dt\\
		&=\hat{G}_U(x,y) + \E_x\left[[X_{\sigma_U-},X_{\sigma_U}]\subset D; \int_{\sigma_U}^\infty\phD(t-\sigma_U,X_{\sigma_U},y)\,dt\right]\\
		&=\hat{G}_U(x,y) + \E_x\left[[X_{\sigma_U-},X_{\sigma_U}]\subset D; \hat{G}_D(X_{\sigma_U},y)\right]. \qquad\qedhere
	\end{align*}
\end{proof}

\begin{lemma}\label{lem:sGf}
For $x, y\in\Rd,\ r>0$,
$$\hat{G}_{rD}(rx,ry) = r^{\alpha-d}\int_{0}^{\infty}\phD(s,x,y)\,ds = r^{\alpha-d}\hat{G}_D(x,y).$$
\end{lemma}
\begin{proof}
By Lemma~\ref{lem:schk},
\begin{align*}
\hat{G}_{rD}(rx,ry)&=\int_{0}^{\infty}\hat{p}_{rD}(t,rx,ry)\,dt=r^{\alpha}\int_{0}^{\infty}\pha_{rD}(r^{\alpha}s,rx,ry)\,ds\\
&=r^{\alpha-d}\int_{0}^{\infty}\phD(s,x,y)\,ds=r^{\alpha-d}\hat{G}_{D}(x,y).
\qquad\qedhere
\end{align*}
\end{proof}
\begin{lemma}\label{cor:pGf}
$\hat{G}_{D}(x,y)>0$ for $x,y\in D$.
\end{lemma}
\begin{proof}
The result follows from Lemma~\ref{lem:phk}.
\end{proof}

\subsection{Harnack lower bound}
{Motivated by Lemma~\ref{lem:mnvG}, we propose the following.}
\begin{definition}\label{def:mvp}
Consider open $U \subset D$. We say that $u: D \to [0,\infty]$ has the mean value property ($mvp$) for $\hat{X}^{{D}}$ inside $U$ if
$$u(x) = \E_x\{[X_{\sigma_B-}, X_{\sigma_B}] \subset D; u(X_{\sigma_B})\}, \quad x \in B \subset\subset U,$$
and we say that $u$ has the supermean value property ($smvp$) if 
$$u(x) \geq \E_x\{[X_{\sigma_B-}, X_{\sigma_B}] \subset D; u(X_{\sigma_B})\}, \quad x \in B \subset\subset U.$$
\end{definition}
\noindent
Without loosing generality in Definition~\ref{def:mvp} we may let $u=0$ on $D^c$. {Note that $\E_x\{[X_{\sigma_B-}, X_{\sigma_B}] \subset D; u(X_{\sigma_B})\}=\E_x[\sigma_B<\sigma_D; u(X_{\sigma_B})]=\E_x[u(\hat{X}^D_{\sigma_B})]$. }

\begin{lemma}\label{lem:green_mvp}
$u(x) = \hat{G}_D(x,y)$ has $smvp$ for $\hat{X}^{{D}}$ on $D$ and $mvp$ for $\hat{X}^{{D}}$ on $D \setminus \{y\}$. 
\end{lemma}
\begin{proof}
Let $y\in \Rd$. From Lemma~\ref{lem:mnvG} we get $smvp$ of $u$.
Then for $B\subset\subset D\setminus \{y\}$, we get $\hat{G}_B(x,y)\le G_B(x,y)=0$, because ${\rm dist}(y,B)>0$.
By Lemma~\ref{lem:mnvG}, 
$\hat{G}_D(x,y)=\E_x \left[[X_{\sigma_B-},X_{\sigma_B}]\subset D; \hat{G}_D(X_{\sigma_B})\right]$.
\end{proof}
\noindent
  The following variant of Harnack inequality is somewhat unusual because infima are used on both sides of 
\eqref{eq:har}.
\begin{proposition}[Harnack lower bound] \label{Harnack_lower}
Let $n>1$. Let $\LL =[x_1, x_2] \cap [x_2, x_3] \cap ... \cap [x_{n-1}, x_n]\subset D$ be a polygonal chain in $D$. Let $\rho \in (0, \text{dist}(\LL , D^c))$.  Let $u \geq 0$ have  the $smvp$ for $\hat{X}^{{D}}$ inside $B(x_i, \rho)$ for each $i=2,3,...,n$.  Let $\varepsilon \in (0,1)$.  Then there is a constant $C=C(d,\alpha, \LL , \rho, \varepsilon)$ such that
\begin{equation}
\label{eq:har}
    \inf_{y \in B(x_n, (1-\varepsilon)\rho)} u(y) \geq C \inf_{x \in B(x_1, \varepsilon\rho)}u(x).
\end{equation}
\end{proposition}
\noindent
The proof of Proposition~\ref{Harnack_lower} is given after the following auxiliary result.
\begin{lemma}
\label{lemma145}
Let $d \in \{1,2,...\}, \alpha \in (0,2), \rho \in (0,\infty), l \in [0, \infty), \varepsilon \in (0,1)$.  Let $D \subset \R^d$ be a (nonempty) domain. Let $x_1, x_2 \in D$ be such that the line segment $L=[x_1, x_2]$ is contained in $D$, ${\rm dist}(L, D^c) \geq \rho$ and $|x_1-x_2| = l$. Assume that $u: D \to [0,\infty]$ has $smvp$ inside $B(x_2, \rho)$ for the shot-down process $\hat{X}^D$ on $D$. Then,
\begin{equation}
    \inf_{y \in B(x_2, (1-\varepsilon) \rho)} u(y) \geq C \inf_{x \in B(x_1, \varepsilon\rho)} u(x),
\end{equation}
where $C=C(d, \alpha, l/\rho, \varepsilon) \in (0,1)$ may be chosen as nonincreasing in $l/\rho$.
\end{lemma}

\begin{proof}
{If $B$ is convex then}
\begin{equation}
\label{eq:eq2}
    u(y) \geq \mathbb{E}_y\{[X_{\tau_B-}, X_{\tau_B} ] {\subset}  D; u(X_{\tau_B})\} \text{ for } y\in B \subset \subset B(x_2, \rho).
\end{equation}
The set $K = \{z \in \Rd: {\rm dist}(z, L) < \rho\}$ is convex (see Figure~\ref{fig:harn_3}). Furthermore, $B(x_1, \rho), B(x_2, \rho) \subset K$ and $K \subset D$. These will be important in applications of \eqref{eq:eq2}. Let $y\in B(x_2, (1-\varepsilon)\rho)$.

 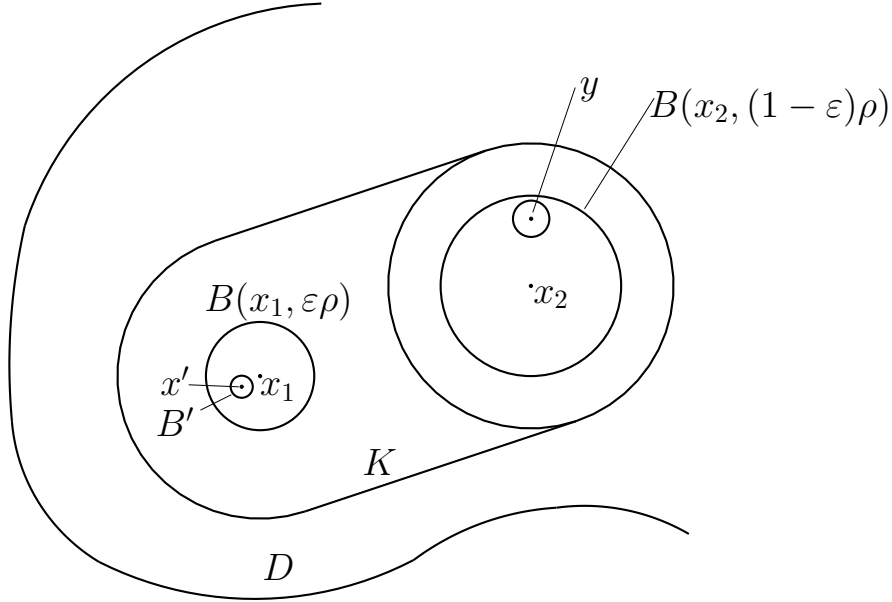
\begin{figure}[H]
   \begin{center}
     \begin{tikzpicture}[scale = 3]
\clip(-5.21890832151443,0.2) rectangle (0.24,2.907323909191167);
\draw [shift={(-3.4,1.2)},line width=0.8pt]  plot[domain=1.892546881191539:5.034139534781332,variable=\t]({1*0.6324555320336759*cos(\t r)+0*0.6324555320336759*sin(\t r)},{0*0.6324555320336759*cos(\t r)+1*0.6324555320336759*sin(\t r)});
\draw [shift={(-2.2,1.6)},line width=0.8pt]  plot[domain=-1.2490457723982544:1.8925468811915387,variable=\t]({1*0.632455532033676*cos(\t r)+0*0.632455532033676*sin(\t r)},{0*0.632455532033676*cos(\t r)+1*0.632455532033676*sin(\t r)});
\draw [line width=0.8pt] (-2.4,2.2)-- (-3.6,1.8);
\draw [line width=0.8pt] (-3.2,0.6)-- (-2,1);
\draw [shift={(-2.2,1.6)},line width=0.8pt]  plot[domain=1.892546881191539:5.034139534781332,variable=\t]({1*0.6324555320336759*cos(\t r)+0*0.6324555320336759*sin(\t r)},{0*0.6324555320336759*cos(\t r)+1*0.6324555320336759*sin(\t r)});
\draw [line width=0.8pt] (-2.2,1.6) circle (0.4cm);
\draw [shift={(-3.424909623632409,1.7392297208275325)},line width=0.8pt]  plot[domain=4.2442265370805385:5.192045550015248,variable=\t]({1*1.5276270273628552*cos(\t r)+0*1.5276270273628552*sin(\t r)},{0*1.5276270273628552*cos(\t r)+1*1.5276270273628552*sin(\t r)});
\draw [shift={(-2.0144910299910896,-0.5593333034704728)},line width=0.8pt]  plot[domain=1.6323106124043163:2.212915400076849,variable=\t]({1*1.1779344426653062*cos(\t r)+0*1.1779344426653062*sin(\t r)},{0*1.1779344426653062*cos(\t r)+1*1.1779344426653062*sin(\t r)});
\draw [shift={(-3.6983723400148714,1.0643007467757388)},line width=0.8pt]  plot[domain=3.2216340175273928:4.1688981818751305,variable=\t]({1*0.8042023944811099*cos(\t r)+0*0.8042023944811099*sin(\t r)},{0*0.8042023944811099*cos(\t r)+1*0.8042023944811099*sin(\t r)});
\draw [shift={(-1.6966773323720798,1.25)},line width=0.8pt]  plot[domain=2.920777801828741:3.230537245382804,variable=\t]({1*2.8144480771256943*cos(\t r)+0*2.8144480771256943*sin(\t r)},{0*2.8144480771256943*cos(\t r)+1*2.8144480771256943*sin(\t r)});
\draw [shift={(-3.0803112903002026,1.4187205912603238)},line width=0.8pt]  plot[domain=1.6044888413985183:2.8241060297778433,variable=\t]({1*1.4341518989145097*cos(\t r)+0*1.4341518989145097*sin(\t r)},{0*1.4341518989145097*cos(\t r)+1*1.4341518989145097*sin(\t r)});
\draw [shift={(-1.9632724705544276,-0.298266484941108)},line width=0.8pt]  plot[domain=1.0449520660228235:1.7051528902233501,variable=\t]({1*0.9229576170950841*cos(\t r)+0*0.9229576170950841*sin(\t r)},{0*0.9229576170950841*cos(\t r)+1*0.9229576170950841*sin(\t r)});
\draw [line width=0.8pt] (-3.4,1.2) circle (0.24cm);
\draw [line width=0.2pt] (-1.6536567439627972,2.4324375754645233)-- (-1.9623908874898721,1.9416153568574614);
\draw [line width=0.2pt] (-2.0035549451347197,2.479892643951931)-- (-2.1916894342225515,1.925657582283358);
\draw [line width=0.2pt] (-3.716452868222108,1.1613301744167097)-- (-3.4943343124915494,1.1523758717054675);
\draw [line width=0.2pt] (-3.6644477838976957,1.0489656174653281)-- (-3.527648939648083,1.117327816970305);
\draw [line width=0.8pt] (-2.198799421550634,1.8976785679259802) circle (0.08059895391676682cm);
\draw (-1.7251601994897274,2.5253034448255027) node[anchor=north west] {\large $B(x_2, (1-\varepsilon)\rho)$};
\draw [line width=0.8pt] (-3.481511768207807,1.152758060255926) circle (0.049cm);
\draw (-3.91,1.11) node[anchor=north west] {\large $B'$};
\draw (-2.9950925146917046,0.924726265574855) node[anchor=north west] {\large $K$};
\draw (-3.89,1.29) node[anchor=north west] {\large $x'$};

\draw (-3.695,1.665) node[anchor=north west] {\large $B(x_1, \varepsilon\rho)$};
\draw (-3.4369382417731367,0.45722771048310745) node[anchor=north west] {\large $D$};
\begin{scriptsize}
\draw [fill=black] (-3.4,1.2) circle (0.2 pt);
\draw[color=black] (-3.32,1.14) node {\large $x_1$};
\draw [fill=black] (-2.2,1.6) circle (0.2 pt);
\draw[color=black] (-2.11,1.55) node {\large $x_2$};
\draw [fill=black] (-2.198799421550634,1.8976785679259802) circle (0.2 pt);
\draw[color=black] (-1.94,2.47) node {\large $y$};
\draw [fill=black] (-3.481511768207807,1.152758060255926) 
circle (0.2 pt);
\end{scriptsize}
     \end{tikzpicture}
   \end{center}
   \caption{Situation in the proof of Lemma~\ref{lemma145} }
   \label{fig:harn_3}
 \end{figure}
\noindent
We only need to prove that $u(y) \geq C \inf_{x \in B(x_1, \varepsilon\rho)}u(x)$ for some appropriate constant $C$. To this end, we consider $B=B(y,\varepsilon\rho/4)$. Let $x'$ be the center of the radius of $B(x,\varepsilon\rho)$ antipodal to $y$. {In particular, $x', x, y$ lie on the same line,} see Figure~\ref{fig:harn_3}.
We define $B'=B(x', \varepsilon\rho/4)$. Thus, $B' \subset B(x_1, \varepsilon\rho)$ and $B'\cap B = \emptyset$.  Clearly, $B, B' \subset K \subset D$.  Then,
\begin{align*}
u(y) &{\ge} \E_y\{[X_{\tau_B-}, X_{\tau_B}] \subset D; u(X_{\tau_B})\} \geq \E_y\{X_{\tau_B} \in B'; u(X_{\tau_B})\} \\
&\geq \mathbb{P}_y(X_{\tau_B} \in B')\inf_{x \in B(x_1,\varepsilon \rho)} u(x).
\end{align*}
\noindent
By the Riesz' formula {for the Poisson kernel of the ball (see, e.g., \cite[(2.6)]{MR1438304}),}
$$\mathbb{P}_y(X_{\tau_B} \in B') = C^d_\alpha\int \limits_{|x-x'| < \varepsilon\rho/4} \frac{(\varepsilon\rho/4)^\alpha}{(|x-y|^2 - (\varepsilon\rho/4)^2)^{\alpha/2}}|x-y|^{-d}\,dx.$$
We have $|x-y| \leq 2\rho+l$ under the integral, hence
\begin{displaymath}
\mathbb{P}_y(X_{\tau_B} \in B') \geq C^d_\alpha (2\rho+l)^{-d-\alpha}(\varepsilon\rho/4)^{d+\alpha} |B(0,1)| = C^d_\alpha  \frac{\omega_d}{d} \left(\frac{\varepsilon}{4}\right)^{\alpha+d} \left(2+\frac{l}{\rho}\right)^{-d-\alpha},
\end{displaymath}
{where $\omega_d$ is the surface measure of the unit sphere in $\Rd$.}
\end{proof}

\begin{proof}[Proof of Proposition~\ref{Harnack_lower}]
$ $\newline

Let $l_i=|x_i-x_{i-1}|, i=2,3,...,n$.  Assume that $\varepsilon \in (0,\frac{1}{2}]$.  By Lemma~\ref{lemma145} and induction,
\begin{align*}
\inf_{y \in B(x_n, (1-\varepsilon)\rho)} {u(y)} &\geq C\left(d,\alpha, \frac{l_n}{\rho}, \varepsilon\right) \inf_{x \in B(x_{n-1}, \varepsilon\rho)}u(x) \geq C\left(d,\alpha, \frac{l_n}{\rho}, \varepsilon\right) \inf_{x \in B(x_{n-1}, (1-\varepsilon)\rho)}u(x)\\
&\geq C\left(d,\alpha, \frac{l_n}{\rho}, \varepsilon\right) \cdot \hdots \cdot C\left(d,\alpha, \frac{l_2}{\rho}, \varepsilon\right) \inf_{x \in B(x_{1}, (1-\varepsilon)\rho)}u(x).
\end{align*}
This proves \eqref{eq:har} for $\varepsilon < \frac{1}{2}$. For $\varepsilon \in (\frac{1}{2}, 1)$, we trivially have \begin{align*}
    \inf_{y \in B(x_n, (1-\varepsilon)\rho)} u(y) &\geq \inf_{x \in B(x_n, \rho/2)}u(x) \\
    &\geq C\left(d,\alpha, \frac{l_n}{\rho}, \frac{1}{2}\right) \cdot \hdots \cdot C\left(d,\alpha, \frac{l_2}{\rho}, \frac{1}{2}\right) \inf_{x\in B(x_1, \rho/2)}u(x) \\
    &\geq C\left(d,\alpha, \frac{l_n}{\rho}, \frac{1}{2}\right) \cdot \hdots \cdot C\left(d,\alpha, \frac{l_2}{\rho}, \frac{1}{2}\right) \inf_{x \in B(x_1,\varepsilon\rho)}u(x).
\end{align*}
The dependence of the estimate on the geometry of $\LL $ and $D$ may be detailed as follows,
$$C\left(d,\alpha, \LL , \rho, \varepsilon\right) = \left[C^d_{\alpha}\frac{\omega_d}{\alpha}\left(\frac{\varepsilon \land \frac{1}{2}}{4}\right)^{\alpha+d}\right]^n \prod \limits_{i=2}^{n} \left(2+\frac{l_i}{\rho}\right)^{-d-\alpha},$$
see the proof of Lemma~\ref{lemma145}.
\end{proof}

\subsection{Sharp estimates of the Green function}\label{subsec:Gf}

\begin{theorem}\label{thm:Green-func-estim}
Let $d \in \{1,2,...\}, \alpha \in (0,2)$ and $d>\alpha$. Let $D\neq \emptyset$ be a domain in $\Rd$. Let $\Gd$ be the Green function of the shot-down process $\hat{X}^{{D}}$. Then,
\begin{equation} \label{eq:Green}
        \Gd(x,y) \approx \frac{\delta_D(x)^{\alpha/2}\delta_D(y)^{\alpha/2}}{r(x,y)^\alpha}|x-y|^{\alpha-d},\quad x,y \in D,
\end{equation}
    where $\delta_D(x) = {\rm dist}(x,D^c)$, $r(x,y)=\delta_D(x) \lor |x-y| \lor \delta_D(y)$. 
\end{theorem}

\noindent
We remark that \eqref{eq:Green} is analogous to the corresponding estimate for the Green function for the killed $\alpha$-stable process in bounded smooth domains (\cite{MR1490808}, \cite{MR1654824}), which has been extended to bounded Lipschitz domains in \cite{MR1991120}:
    \begin{equation} \label{eq:Green_killed}
        G_D(x,y) \approx \frac{\delta_D(x)^{\alpha/2}\delta_D(y)^{\alpha/2}}{r(x,y)^\alpha}|x-y|^{\alpha-d},\quad x,y \in D.
\end{equation}
Therefore Theorem~\ref{thm:Green-func-estim} in fact asserts  the comparability:
$$
 \Gd(x,y) \approx  G_D(x,y),\quad x,y\in 
\Rd.
$$
We emphasize that connectedness plays an important role in the potential theory of $\hat{X}^D$.  Accordingly, our proof uses the same old-fashioned  Harnack chain argument
as for classical harmonic functions \cite{MR1741527}, whereas the Harnack inequality for $\alpha$-harmonic functions  has a more flexible {proof and} statement;  see {\cite{MR1438304}}.
\newline

\begin{proof}[Proof of Theorem~\ref{thm:Green-func-estim}]

We have $\Gd(x,y) \leq G_D(x,y)$ for $x,y \in \Rd$, hence we only need to prove the lower bound:
\begin{equation} \label{Green1}
    \Gd(x,y) \geq c \frac{\delta_D(x)^{\alpha/2}\delta_D(y)^{\alpha/2}}{r(x,y)^\alpha}|x-y|^{\alpha -d}, \quad x,y \in D.
\end{equation}
Let $r_0>0$ be a localization radius of $D$. For $x\in D$, there is $\underline{x}\in \partial D$ such that $\delta_D(x)=|x-\underline{x}|$, and the ball $B(\bar{x},r_0/2)$ tangent to $\partial D$ at $\underline{x}$,
see Figure~\ref{fig:Green1}. {For a ball $B=B(z,r)$ and $\varepsilon > 0$, we denote $\varepsilon B = B(z,\varepsilon r)$.} 
  \begin{figure}[H]
   \begin{center}
\begin{tikzpicture}[line cap=round,line join=round,>=triangle 45,x=1cm,y=1cm]
\draw [line width=0.8pt, shift={(2,-2)}]  plot[domain=0.7811169899487521:2.351913316743649,variable=\t]({1*0.8284378614529069*cos(\t r)+0*0.8284378614529069*sin(\t r)},{0*0.8284378614529069*cos(\t r)+1*0.8284378614529069*sin(\t r)});
\draw [line width=0.8pt, shift={(2,2)}]  plot[domain=3.9357750941215457:5.506571420916575,variable=\t]({1*0.8284723287791202*cos(\t r)+0*0.8284723287791202*sin(\t r)},{0*0.8284723287791202*cos(\t r)+1*0.8284723287791202*sin(\t r)});
\draw [line width=0.8pt, shift={(0,0)}]  plot[domain=0.781759436897226:5.499560482382657,variable=\t]({1*2*cos(\t r)+0*2*sin(\t r)},{0*2*cos(\t r)+1*2*sin(\t r)});
\draw [line width=0.8pt] (0,-1.3) circle (0.7cm);
\draw (-1.1,0.20) node[anchor=north west] {\large $B(\bar{x}, r_0/2)$};
\begin{scriptsize}
\draw [line width=0.8pt, fill=black] (0,-1.35) circle (0.6pt);
\draw[color=black] (0.26566251211,-1.13432867596) node {\large $\bar{x}$};
\draw [fill=black] (0,-2.00) circle (0.6pt);
\draw[color=black] (0.2656119117646,-2.23870826209) node {\large $\underline{x}$};
\draw [fill=black] (0,-1.7437958613) circle (0.6pt);
\draw[color=black] (0.26566251211,-1.64437958613) node {\large $x$};
\draw [fill=black] (1.9,0.67) circle (0.6pt);
\draw (1.95,1) node[anchor=north west] {\large $y$};
\end{scriptsize}
\clip(2.59,-3)rectangle(6,3);
\draw [line width=0.8pt, rotate around={90:(3.5,0)}] (3.5,0) ellipse (1.9856938509374418cm and 1.3011456757991278cm);
\end{tikzpicture}

   \end{center}
   \caption{Situation in the proof of Theorem~\ref{thm:Green-func-estim}}
   \label{fig:Green1}
 \end{figure}
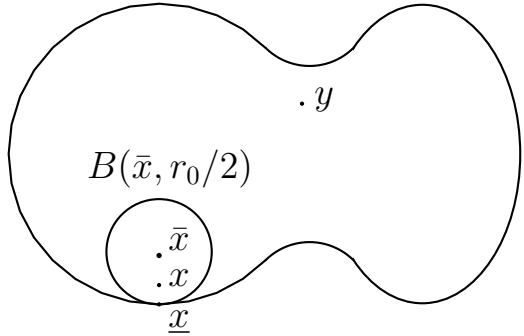
\noindent
Let us consider several cases.

$\mathbf{Case \enspace 1:}$ Assume $\varepsilon \in (0,1)$ and there are $r>0, \text{ } z \in D$ such that $B=B(z,r)\subset D \text{ and }x,y \in \varepsilon B$. Then,
$$\Gd(x,y) \geq \hat{G}_B(x,y) = G_B(x,y) \approx |x-y|^{\alpha - d}.$$
This follows from \eqref{eq:Green_killed}, because $\delta_D(x) \approx \delta_D(y) \approx r(x,y),$ and so \eqref{Green1} also hold in the case considered.

\smallskip

$\mathbf{Case \enspace 2:}$ Let $K = \{ x\in D: \delta_D(x) \geq r_0/8\}$. $K$ is connected and by the Harnack lower bound, Lemma~\ref{lem:green_mvp} and $Case \text{ } 1$, we get
$$\Gd (x,y)\geq c, \quad x,y \in K.$$
{The constant depends on $d$, $\alpha$, the localization radius  $r_0$ of $D$ and its diameter, since the length of the Harnack chain in Proposition~\ref{Harnack_lower} depends on the global geometry, see, e.g., the discussion in \cite{MR1741527} and \bf{Case 6} below.}

$\mathbf{Case \enspace 3:}$ Let $\delta_D(x) < r_0/8, \delta_D(y)>r_0/4$.
We consider the ball $B_x=B(\tilde{x},r_0/4) \subset D$ tangent to $\partial D$ at $\underline{x}$, of radius $r_0/4$,
and the ball $B = B(\tilde{x}, r_0/8) \subset D$ tangent to $\partial D$ at $\overline{x}$.
Note that $\underline{x} \in \partial B$, and $y \notin B$. 
Furthermore, we consider a ball $B'\subset B_x$ of radius $r_0/8$ tangent to $\partial B_x$ at $\bar{x}$. The situation is shown in Figure~\ref{fig:Green2}.
 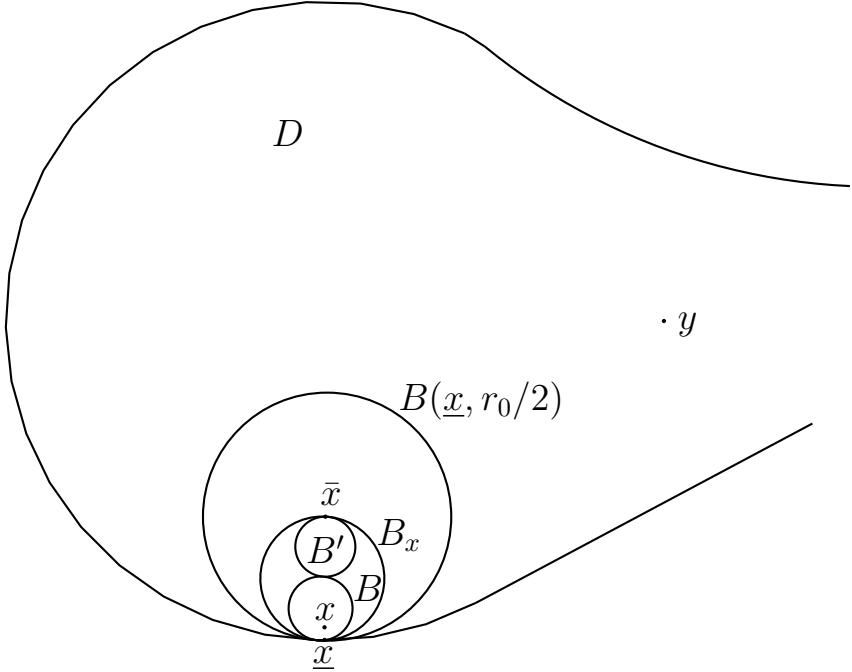
\begin{figure} [H]
   \begin{center}
\begin{tikzpicture}[scale = 2.5]
\clip(-14.79853890462575,3.65) rectangle (-7.350299106689115,7.452172688625365);
\draw [line width=0.8pt] (-12.791191176475985,4.458708262091341) circle (0.6601331520776625cm);
\draw [shift={(-12.816457806519963,4.129610188526304)},line width=0.8pt]  plot[domain=1.494171264125722:4.635763917715515,variable=\t]({1*0.3300665760388314*cos(\t r)+0*0.3300665760388314*sin(\t r)},{0*0.3300665760388314*cos(\t r)+1*0.3300665760388314*sin(\t r)});
\draw [shift={(-12.816457806519963,4.129610188526304)},line width=0.8pt]  plot[domain=-1.6474213894640712:1.4941712641257219,variable=\t]({1*0.3300665760388311*cos(\t r)+0*0.3300665760388311*sin(\t r)},{0*0.3300665760388311*cos(\t r)+1*0.3300665760388311*sin(\t r)});
\draw [shift={(-12.825879796214096,3.970013201583418)},line width=0.8pt]  plot[domain=1.4775890786931656:4.619181732282959,variable=\t]({1*0.17024003933829976*cos(\t r)+0*0.17024003933829976*sin(\t r)},{0*0.17024003933829976*cos(\t r)+1*0.17024003933829976*sin(\t r)});
\draw [shift={(-12.825879796214096,3.970013201583418)},line width=0.8pt]  plot[domain=-1.6640035748966273:1.4775890786931654,variable=\t]({1*0.17024003933829931*cos(\t r)+0*0.17024003933829931*sin(\t r)},{0*0.17024003933829931*cos(\t r)+1*0.17024003933829931*sin(\t r)});
\draw [shift={(-12.800613166170116,4.299111275148455)},line width=0.8pt]  plot[domain=1.5118286315888205:4.653421285178614,variable=\t]({1*0.15987486397506043*cos(\t r)+0*0.15987486397506043*sin(\t r)},{0*0.15987486397506043*cos(\t r)+1*0.15987486397506043*sin(\t r)});
\draw [shift={(-12.800613166170116,4.299111275148455)},line width=0.8pt]  plot[domain=-1.6297640220009733:1.5118286315888203,variable=\t]({1*0.15987486397506054*cos(\t r)+0*0.15987486397506054*sin(\t r)},{0*0.15987486397506054*cos(\t r)+1*0.15987486397506054*sin(\t r)});
\draw (-12.470259562662806,5.2338439572306905) node[anchor=north west] {\large $B(\underline{x}, r_0/2)$};
\draw (-12.96,4.42) node[anchor=north west] {\large $B'$};
\draw [line width=0.8pt] (-12,4) -- (-10.2104149617045,4.9544453537576);
\draw [line width=0.8pt] (-12.060203050406658,7.0305882768963)-- (-11.95307461436431,6.960883641881606);
\draw [shift={(-12.8,5.5)},line width=0.8pt]  plot[domain=1.1205640490127573:5.202346306638418,variable=\t]({1*1.7*cos(\t r)+0*1.7*sin(\t r)},{0*1.7*cos(\t r)+1*1.7*sin(\t r)});
\draw [shift={(-9.862175202331258,9.516906408040526)},line width=0.8pt]  plot[domain=4.026751328926791:4.66728428867869,variable=\t]({1*3.302289014057198*cos(\t r)+0*3.302289014057198*sin(\t r)},{0*3.302289014057198*cos(\t r)+1*3.302289014057198*sin(\t r)});
\begin{scriptsize}
\draw [fill=black] (-12.8,4.458708262091341) circle (0.24pt);
\draw[color=black] (-12.77,4.5801503849384) node {\large $\bar{x}$};
\draw [fill=black] (-12.805,3.80300512114961268) circle (0.24pt);
\draw[color=black] (-12.810559385406548,3.70408) node {\large $\underline{x}$};
\draw [fill=black] (-11,5.5) circle (0.24pt);
\draw[color=black] (-10.874218315429,5.483536558950585) node {\large $y$};
\draw [fill=black] (-12.805,3.87) circle (0.24pt);
\draw[color=black] (-12.8,3.95) node {\large $x$};
\draw[color=black] (-12.575,4.08) node {\large $B$};
\draw[color=black] (-12.41,4.36) node {\large $B_x$};
\draw[color=black] (-13,6.5) node {\large $D$};
\end{scriptsize}
\end{tikzpicture}
   \end{center}
   \caption{Proof of \ref{thm:Green-func-estim}: $Case \enspace 3$}
\label{fig:Green2}
 \end{figure}
\noindent
By $Case \enspace 2$ for $z \in B'$ we have $\Gd(z,y) \geq c$. Since $B_x$ is a convex subset of $D$,
\begin{align*}
 \Gd(x,y) &= \E_x\big\{ \Gd(X_{\tau_B}, y); [X_{\tau_B-}, X_{\tau_B}] \in D \big\} \geq \E_x\big\{ \Gd(X_{\tau_B}, y);  X_{\tau_B} \in B' \big\} \\
 &\geq c \text{ } \PP_x \big\{ X_{\tau_B} \in B' \big \} \geq c' \int \limits_{B'} \frac{\big ((r_0/8)^2 - |x-\tilde{x}|^2\big)^{\alpha/2}}{\big(|z-x|^2 - (r_0/8)^2\big)^{\alpha/2}}|z-x|^{-d}\,dz \\
 &\geq c''(\delta_D(x)/r_0)^{\alpha/2} \approx \delta_D(x)^{\alpha/2},
\end{align*}
which proves \eqref{Green1} in this case.

\smallskip

$\mathbf{Case \enspace 4:}$ If $\delta_D(y) < r_0/8, \text{ } \delta_D(x) > r_0/4$, then we recall that $\Gd(x,y) = \Gd(y,x)$, and \eqref{Green1} follows from $Case \enspace 3$.

\smallskip

$\mathbf{Case \enspace 5:}$ Let $\delta_D(x) < r_0/8, \text{ } \delta_D(y) < r_0/8$, and $|x-y|>r_0$. As in $Case \enspace 3$ we consider a ball $B \subset D$ of radius $r_0/8$ and center $\tilde{x}$ tangent to $\partial D$ at $\underline{x}$. We also consider the ball $B'$ from $Case \enspace 3$, see Figure~\ref{fig:Green2}. We denote them by $B(x)$ and $B'(x)$, respectively, and construct analogous balls $B(y)$ and $B'(y)$ for $y$. Note that $y \notin B_x$. By convexity of $B_x$,
$$
 \Gd(x,y) = \E_x\big\{ \Gd(X_{\tau_{B(x)}}, y); [X_{\tau_{B(x)-}}, X_{\tau_{B(x)}}] \subset D \big\} \geq \PP_x \big\{ X_{\tau_{B(x)}} \in B'(x) \big \}\!\! \inf_{z \in B'(x)}\Gd(z,y).
$$
For $z \in B'(x)$, we have $z \notin B(y)$, and we get
$$
 \Gd(z,y) = \E_y\big\{ \Gd(z, X_{\tau_{B(y)}}); [X_{\tau_{B(y)-}}, X_{\tau_{B(y)}}] \subset D \big\} \geq \PP_y \big\{ X_{\tau_{B(y)}} \in B'(y) \big \}\!\! \inf_{w \in B'(y)}\Gd(z,w).
$$
Therefore $\Gd(x,y) \geq c \delta_D(x)^{\alpha/2}\delta_D(y)^{\alpha/2}$, see $Case \enspace 3$.

\smallskip

$\mathbf{Case \enspace 6:}$ Let $\delta_D(x)<r_0/8, \delta_D(y) < r_0/8, |x-y| < r_0$. This case locally repeats the arguments of $Case \enspace 3-5$, but to handle it
we need the following Harnack chain of balls.
\newline
\noindent
\begin{addmargin}[1em]{1em}
If $\varepsilon > 0, x_1, x_2$ belong to $D$, $\delta(x_j)>\varepsilon$ and $|x_1-x_2| < 2^k\varepsilon$, then there is a Harnack chain from $x_1$ to $x_2$ of length $Mk$, where $M$ is a constant depending only on $D$. Furthermore, for each ball $B$ in the chain, its radius is not smaller than $M^{-1}\min(\text{dist}(x_1,B),\text{dist}(x_2,B))$ (see \cite{MR676988} for details).
\newline
By Harnack lower bound, if $x_1, x_2$ are as above, then every nonnegative function with $smpv$ in $D$ satisfies:
$$\inf_{y \in (1-\varepsilon)B_2}u(y) \geq c^k \inf_{x \in \varepsilon B_1} u(x).$$
\end{addmargin}
We now sketch the argument for $Case \enspace 6:$ If $|x-y| \leq N[\delta_D(x) \lor \delta_D(y)]$, then by the Harnack chain and $Case \enspace 1$ with radii $s \approx \delta_D(x) \lor \delta_D(y),$ we have $G(x,y) \geq c|x-y|^{\alpha - d}$. Take large $N$ and $r_0 > |x-y| > N[\delta_D(x) \lor \delta_D(y)] $. Assume $\delta_D(x) \leq \delta_D(y)$ and take $\tilde{x} \in D$ on the ray through $\underline{x}, x$ such that $\delta(\tilde{x}) = \frac{1}{3}|x-y|.$ Similarly, define $\tilde{y}$. Then
$
|\tilde{x}-\tilde{y}|, \delta(\tilde{x}), \delta(\tilde{y}) \approx r(x,y)=r.
$
We have by Harnack lower bound and by considering the tangent balls $B_x, B_y$ as in $Case \enspace 5$ and  $Case \enspace 4$:
\begin{align*}
\Gd(x,y) &\geq c\left( \frac{\delta_D(x)}{r} \right)^{\alpha/2}\Gd(\tilde{x},y) \geq c\left( \frac{\delta_D(x)}{r} \right)^{\alpha/2} \left( \frac{\delta_D(y)}{r} \right)^{\alpha/2} \Gd(\tilde{x},\tilde{y}) \\
&\geq c\frac{\delta_D(x)^{\alpha/2} \delta_D(y)^{\alpha/2}}{r(x,y)^\alpha} r(x,y)^{\alpha -d}.
\qquad\qedhere
\end{align*}
\end{proof}


\end{document}